\documentclass{conm-p-l}
\usepackage{sabbah_Fourier-local}

\begin{document}
\title[An explicit stationary phase formula]
{An explicit stationary phase formula for the~local formal Fourier-Laplace~transform}

\begin{abstract}
We give an explicit formula (i.e., an explicit expression for the formal stationary phase formula) for the local Fourier-Laplace transform of a formal germ of meromorphic connection of one complex variable with a possibly irregular singularity.
\end{abstract}



\author[C\ptbl Sabbah]{Claude Sabbah}
\address{UMR 7640 du CNRS\\
Centre de Math\'ematiques Laurent Schwartz\\
\'Ecole polytechnique\\
F--91128 Palaiseau cedex\\
France}
\email{sabbah@math.polytechnique.fr}
\urladdr{http://www.math.polytechnique.fr/~sabbah}

\maketitle

\section*{Introduction}
We will denote by $\Clt$ (\resp $\Chlt$) the field $\CC\lcr t\rcr[\tm]$ (\resp $\CC\{t\}[\tm]$) of formal (\resp convergent) Laurent series of the variable $t$, equipped with its usual derivation~$\partial_t$.

Let $M$ be a finite dimensional $\Clt$-vector space with a connection $\nabla$. The local formal Laplace transform $\cF^{(0,\infty)}$ (also called Fourier transform in the literature) was introduced in \cite{B-E04,Garcia04} by analogy with the $\ell$-adic local Fourier transform considered in \cite{Laumon87}. One way to produce it is to choose a free $\CC[t,\tm]$-module $\cM$ of finite rank equipped with a connection $\nabla$ having poles at most at $t=0$ and $t=\infty$, with a regular singularity at infinity, and such that $(\Clt\otimes_{\CC[t,\tm]}\cM,1\otimes\nabla)=(M,\nabla)$. Considering~$\cM$ as a $\CC[t]\langle\partial_t\rangle$-module, its (global) Laplace transform $F\cM$ is the same $\CC$-vector space equipped with the $\CC[\tau]\langle\partial_\tau\rangle$-structure defined by the correspondence $\tau=\partial_t$, $\partial_\tau=-t$. Tensoring with $\CC[\tau,\tau^{-1}]$ gives a $\CC[\tau,\tau^{-1}]\langle\tau\partial_\tau\rangle$-module $F\cM[\tau^{-1}]$, and renaming $\theta=\tau^{-1}$ (and setting $\theta\partial_\theta=-\tau\partial_\tau$), we regard $F\cM[\tau^{-1}]$ as a $\CC[\theta,\theta^{-1}]\langle\theta\partial_\theta\rangle$-module. Lastly, we define $\cF^{(0,\infty)}M$ as $\Clth\otimes_{\CC[\theta,\theta^{-1}]}F\cM[\tau^{-1}]$, equipped with its natural connection. This does not depend of the choices made.

The previous transform corresponds to using the kernel $e^{-t/\theta}$, and is also denoted by $\cF^{(0,\infty)}_-$. Its inverse transform is denoted by $\cF^{(\infty,0)}_+$. There are also pairs of inverse transforms $(\cF^{(s,\infty)}_\pm,\cF^{(\infty,s)}_\mp)$ for any $s\in\CC$ and $(\cF^{(\infty,\infty)}_\pm,\cF^{(\infty,\infty)}_\mp)$. If we denote by $F_\pm$ the algebraic Laplace transform with kernel $e^{\pm t\tau}$ acting on a $\CC[t]\langle\partial_t\rangle$-module $\cM$, the local formal stationary phase formula of \cite{B-E04,Garcia04} relies the formalization of the Laplace transform of $\cM$ at each of its singularities ($0$,~$\wh s\in\CC^*$,~$\wh\infty$) with the local formal Laplace transforms of $\cM$ itself at its singularities $0$,~$s\in\CC^*$,~$\infty$. Decomposing with respect to slopes at infinity gives the following diagram:
\[\def\labelstyle{\scriptstyle}
\xymatrix@R=1.5cm{
\ar@<3mm>@/_1pc/[d]_(.35){F_-}
&\ar@/_1pc/[d]_(.35){\cF^{(0,\infty)}_-\!\!}\cM_0
&\ar@/_1pc/[d]_(.35){\oplus\cF^{(s,\infty)}_-\!\!}\oplus_{s\neq0}\cM_s
&\ar@/_1pc/[d]_(.35){\cF^{(\infty,\infty)}_-\!\!}\cM_\infty^{>1}
&\ar@/_1pc/[d]_(.35){\oplus\cF^{(\infty,\wh s)}_-\!\!}\cM_\infty^{=1}
&\ar@/_1pc/[d]_(.35){\cF^{(\infty,0)}_-\!\!}\cM_\infty^{<1}&\\
&\wh \cM_\infty^{<1}\ar@/_1pc/[u]_(.35){\!\cF^{(\infty,0)}_+}
&\wh \cM_\infty^{=1}\ar@/_1pc/[u]_(.35){\!\oplus\cF^{(\infty,s)}_+}
&\wh \cM_\infty^{>1}\ar@/_1pc/[u]_(.35){\!\cF^{(\infty,\infty)}_+}
&\oplus_{\wh s\neq0}\wh \cM_{\wh s}\ar@/_1pc/[u]_(.35){\!\oplus\cF^{(\wh s,\infty)}_+}
&\wh \cM_0\ar@/_1pc/[u]_(.35){\!\cF^{(0,\infty)}_+}
&\ar@<3mm>@/_1pc/[u]_(.35){F_+}
}
\]

Starting from a given $\Clt$-vector space $M$ with connection, the explicit computation of $\cF^{(0,\infty)}M$ (or of the other local transforms) by means of equations can be cumbersome (see however a simple example in \S\ref{sec:example}). In this article, we show (Theorem \ref{th:main}) how to pass explicitly from the Turrittin-Levelt decomposition of~$M$ to that of $\cF^{(0,\infty)}M$ (and similarly for the other local transforms). We note that such formulas were already mentioned by G\ptbl Laumon \cite[\S2.6.3]{Laumon87} and attributed to B\ptbl Malgrange, as a motivation for similar formulas in the $\ell$-adic situation. Such formulas do not seem to be explicitly written (nor proved) in the literature (see the remark below, however). In this article, we provide a geometric method for the proof.

After some preliminaries fixing notation (\S\ref{sec:prelim}), we introduce elementary formal meromorphic connections and give their main properties (\S\ref{sec:elemform}) and we recall in \S\ref{sec:turrittin} the refined Turrittin-Levelt decomposition, as obtained in \cite{B-E04}. The results of \S\ref{sec:exptwist} are mainly given in this article for a better understanding of the formulas in \S\ref{sec:lapalceloc}, but are not directly used in the proof of the main theorem (Theorem \ref{th:main}). Lastly, in \S\ref{subsec:rigidity}, we give some consequences of the theorem of preservation of the index of rigidity, proved in \cite{B-E04} following the proof of \cite{Katz96} in the $\ell$-adic case.

\begin{remarque*}
After this article was written up, Ricardo Garc\'\i a L\'opez pointed out to me the preprint \cite{Fu07}, where a similar calculation for local Fourier transforms is done in the $\ell$-adic case. The formulas we give in \S\ref{sec:lapalceloc} are the complex analogues of that in \cite{Fu07}. Let us also mention the article \cite{Fang07}, the results of which where obtained approximately at the same time as the results of the present article, and in an independent way. The methods developed in \cite{Fang07} are of a more computational flavour, while those in the present article emphasize the geometry of the formulas; we try in particular to explain in a more general context (\cf \S\ref{sec:exptwist}) the reason why the Legendre transform occurs in the formula for $\wh\varphi$ of Theorem \ref{th:main}.
\end{remarque*}

\subsubsection*{Acknowledgements}
The results of this article came out from discussions with Ricardo Garc\'\i a L\'opez on the one hand, and with C\'eline Roucairol on the other hand. I~thank both of them. I also thank Hélène Esnault for useful conversations on this subject during Lê's fest.

\section{Preliminaries}\label{sec:prelim}

\subsection{Operations on vector spaces with connection}
We work in the abelian category of $\Clt$-vector spaces with a connection usually denoted by $\nabla$. This category has tensor products and duality in a natural way. We simply denote the action of $\nabla_{\partial_t}$ as $\partial_t$. We usually omit, when there should be no confusion, the subscript denoting the field when using tensor products. It will be convenient to denote by $\bun$ the field itself with its connection~$d$.

Let $\rho\in u\Cu$ with valuation $p\geq1$. We regard $\rho$ as a morphism of degree $p$ from the formal disc with coordinate $u$ to the formal disc with coordinate $t$ through the correspondence $\rho:\Ct\to\Cu$, $t\mto\rho(u)$.

Let $M$ be a finite dimensional $\Clt$-vector space equipped with a connection~$\nabla$. The pull-back $\rho^+M$ is the vector space $\rho^*M=\Clu\otimes_{\Clt}M$ equipped with the pull-back connection $\rho^*\nabla$ defined by $\partial_u(1\otimes m)=\rho'(u)\otimes\partial_tm$.

Let $N$ be a $\Clu$-vector space with connection. The push-forward $\rho_+N$ is defined as follows:

\eqitem{enum:a}
the $\Clt$-vector space $\rho_*N$ is the $\CC$-vector space $N$ equipped with the structure of $\Clt$-vector space given by $f(t)\cdot m\defin f(\rho(u))m$,

\eqitem{enum:b}
the action of $\partial_t$ is that of $\rho'(u)^{-1}\partial_u$.

\smallskip
The projection formula holds:
\begin{equation}\label{eq:proj}
\rho_+(N\otimes_{\Clu}\rho^+M)\simeq\rho_+N\otimes_{\Clt}M.
\end{equation}

If $N^*$ denotes $\Hom_{\Clu}(N,\Clu)$ with its natural connection, we have $\rho_+(N^*)\simeq(\rho_+N)^*$.
Therefore we also have the projection formula
\[
\rho_+\Hom_{\Clu}(N,\rho^+M)\simeq\Hom_{\Clt}(\rho_+N,M).
\]

Let $\varphi\in\Clu$. We denote by $\cE^\varphi$ the rank-one vector space $\Clu$ equipped with the connection $\nabla=d+d\varphi$, \ie such that $\nabla_{\partial_u}1=\varphi'$. We have $\cE^\varphi\simeq\cE^\psi$ if and only if $\varphi\equiv\psi\bmod\Cu$.
 
\subsection{Moderate nearby cycles}
Let $X$ a smooth complex algebraic variety and let $f:X\to\Afu$ be a function on $X$, defining a reduced divisor $D=f^{-1}(0)$. Let~$\cM$ be a holonomic left $\cD_X$-module such that $\cM=\cO_X(*D)\otimes_{\cO_X}\cM$ (we then say that $\cM$ is localized away from $D$). The (moderate) nearby cycles module $\psi_f\cM$ is a holonomic left $\cD_X$-module supported on $D$ equipped with an automorphism $T:\psi_f\cM\to\psi_f\cM$.

Let $\pi:X'\to X$ be a proper modification inducing an isomorphism $X'\moins\pi^{-1}(D)\to X\moins D$ and let us set $f'=f\circ\pi$, $D'=\pi^{-1}(D)=f^{\prime-1}(0)$. If $\cM'$ is a holonomic left $\cD_{X'}$-module localized away from $D'$, we will denote by $\pi^0_+\cM'$ the holonomic $\cD_X$-module $\cH^0\pi_+\cM'(*D)$. As a $\cO_X(*D)$-module, it is equal to $\pi_*\cM'$.

As $\pi$ is proper, we have (see \eg \cite{MSaito86,M-S86b})
\begin{equation}\label{eq:psipi}
\psi_f\pi^0_+\cM'\simeq\cH^0\pi_{|D',+}\psi_{f'}\cM',\quad \cH^j\pi_{|D',+}\psi_{f'}\cM'=0\text{ if }j\neq0.
\end{equation}

Let us assume that $X$ is the affine space $\AA^2$ with coordinates $(x_1,x_2)$ and that $f(x_1,x_2)=x_1^{m_1}x_2^{m_2}=x^m$ with $m_1\in\NN$ and $m_2\in\NN^*$. Let us set $D_1=\{x_2=0\}$ with coordinate $x_1$ and $D=|\{x^m=0\}|$. Let $\cR$ be a locally free $\cO_X(*D)$-module of finite rank with a flat connection having regular singularitites along $D$. Then $\cR$ is also a regular holonomic $\cD_X$-module. The $\cD_X$-module $\psi_f\cR$ is supported on $D$. Moreover, if $m_1\neq0$, $(\psi_f\cR)[x_1^{-1}]$ is supported on $D_1$ and is the direct image (in the sense of left $\cD_X$-modules) by the inclusion $D_1\hto\AA^2$ of a regular holonomic $\cD_{D_1}$-module localized (and smooth) away from $\{x_1=0\}$. We will not distinguish between both, according to Kashiwara's equivalence.

\begin{proposition}[{\cf \cite[Lemma~III.4.5.10]{Bibi97}}]\label{prop:bibi}
With the previous setting, for any $\lambda\in\CC^*$,
\begin{enumerate}
\item\label{prop:bibi1}
if $n\in(\NN^*)^2$, $\psi_f(\cE^{\lambda/x^n}\otimes\cR)=0$
\item\label{prop:bibi2}
if $n_1\in\NN^*$, $\psi_f(\cE^{\lambda/x_1^{n_1}}\otimes\cR)$ is supported on $D_1$; it is isomorphic to \hbox{$\cE^{\lambda/x_1^{n_1}}\otimes\big((\psi_f\cR)[x_1^{-1}]\big)$} (with monodromy).\qed
\end{enumerate}
\end{proposition}

Let us note that \ref{prop:bibi}\eqref{prop:bibi2} is stated in a weaker way in \loccit, but the argument given in the proof gives \ref{prop:bibi}\eqref{prop:bibi2}.

\section{Elementary formal meromorphic connections}\label{sec:elemform}

Let $M$ be a finite dimensional $\Clt$-vector space equipped with a connection~$\nabla$. The classical Turrittin-Levelt theorem asserts that, after a suitable ramification \hbox{$\rho:u\mto t=u^p$}, the pull-back $\rho^+M$ can be decomposed into elementary formal connections $\cE^\varphi\otimes R_\varphi$, where $\varphi\in\Clu$ and $R_\varphi$ has a regular singularity. Moreover, it is known that $M$ itself can be decomposed according to the slopes of its Newton polygon (see, \eg \cite{Malgrange91}).

In the next section, we refine the decomposition with respect to slopes, in order to keep as much information as possible from the Turrittin-Levelt decomposition, by using the elementary formal connections that we define now.

\begin{definition}[Elementary formal connections]\label{def:elem}
Given $\rho\in u\Cu$, $\varphi\in\Clu$ and a finite dimensional $\Clu$-vector space~$R$ with regular connection $\nabla$, we define the associated elementary finite dimensional $\Clt$-vector space with connection by
\[
\El(\rho,\varphi,R)=\rho_+(\cE^\varphi\otimes R).
\]
\end{definition}

If $p$ denotes the order of $\rho$, $q$ the order of the pole of $\varphi$ and $r$ the rank of $R$, then
\begin{itemize}
\item
$\El(\rho,\varphi,R)$ has only one slope, which is equal to $q/p$,
\item
the irregularity number $\irr_0\El(\rho,\varphi,R)$ is equal to $qr$,
\item
the rank of $\El(\rho,\varphi,R)$ is equal to $pr$.
\end{itemize}
Up to isomorphism, $\El(\rho,\varphi,R)$ only depends on~$\varphi\bmod\Cu$. Standard results on regular formal meromorphic connections then show that, up to isomorphism, any elementary vector space $\El(\rho,\varphi,R)$ with connection is defined over the field of convergent series $\Chlt$. Let us also note that giving $R$ is equivalent to giving a finite dimensional $\CC$-vector space equipped with an automorphism $T$.

Let us first distinguish the isomorphism classes of the elementary finite dimensional $\Clt$-vector spaces with connection.

\begin{lemme}\label{lem:autom}
Assume that $\rho'(0)\neq0$ (\ie $\rho$ is an automorphism of the formal disc). Then $\El(\rho,\varphi,R)\simeq \El(\Id,\psi,S)$ if and only if $\psi\circ\rho\equiv\varphi\bmod\Cu$ and $S\simeq R$.
\end{lemme}

\begin{proof}
Let us denote by $\lambda$ the reciprocal series of $\rho$, that is, $\rho\circ\lambda(t)=t$, $\lambda\circ\rho(u)=u$. We use that, $\rho_+=\lambda^+$. Then
$$
\El(\rho,\varphi,R)=\lambda^+(\cE^\varphi\otimes R)=\cE^{\varphi\circ\lambda}\otimes\lambda^+R\simeq\cE^{\varphi\circ\lambda}\otimes R,
$$
hence, tensoring with $\cE^{-\psi}$, we find $S\simeq\cE^{\varphi\circ\lambda-\psi}\otimes R$, and the lemma follows easily.
\end{proof}

Let us come back to the general situation where $\rho$ has degree $p\geq1$. We deduce from the previous lemma that any elementary vector space $\El(\rho,\varphi,R)$ with connection is isomorphic to an elementary vector space $\El([u\mto u^p],\psi,R)$ with connection for a suitable $\psi$. More precisely, if $u\mto\lambda(u)$ is a formal automorphism (\ie $\lambda'(0)\neq0$) then
\begin{equation}\label{eq:rholambda}
\El(\rho,\varphi,R)\simeq\El(\rho\circ\lambda,\varphi\circ\lambda,R).
\end{equation}

Let us denote by $\rho$ the map $u\mto u^p$ and by $\mu_\zeta$ the map $u\mto\zeta u$.

\begin{lemme}\label{lem:pull-back}
For any $\varphi\in\Clu$, we have
\[
\rho^+\rho_+\cE^\varphi=\tbigoplus_{\zeta^p=1}\cE^{\varphi\circ\mu_\zeta}.
\]
\end{lemme}

\begin{proof}
We choose a $\Clu$-basis $e$ of $\cE^\varphi$ and assume for simplicity that $\varphi\in\um\CC[\um]$. Then the family $e,ue,\dots,u^{p-1}e$ is a $\Clt$-basis of $\rho_+\cE^\varphi$. Set $e_k=u^{-k}\otimes_{\Clt}u^ke$. Then the family $\bme=(e_0,\dots,e_{p-1})$ is a $\Clu$-basis of $\rho^+\rho_+\cE^\varphi$. Let us decompose $u\varphi'(u)=\sum_{j=0}^{p-1}u^j\psi_j(u^p)$ with $\psi_j\in\CC[\tm]$ for any~$j\geq1$ and $\psi_0\in\tm\CC[\tm]$. Let $\rP$ denote the permutation matrix defined by $\bme\cdot \rP=(e_1,\dots,e_{p-1},e_0)$. Using \eqref{enum:b}, we find that
\[
u\partial_ue_k=\sum_{j=0}^{p-k-1}u^j\psi_je_{k+j}+\sum_{j=p-k}^{p-1}u^j\psi_je_{k+j-p},
\]
that is,
\[
u\partial_u\bme=\bme\cdot\Big[\sum_{j=0}^{p-1}u^j\psi_j\rP^j\Big].
\]
The result is obtained by diagonalizing $\rP$.
\end{proof}

\begin{remarque}\label{rem:partialdecomp}
Let us decompose $p$ as a product $p=p'd$ and $\rho:u\mto u^p$ as $\rho_d\circ\rho'$ correspondingly. Then $\rho_d^+\rho_+\cE^\varphi$ decomposes as $\bigoplus_{k=0,\dots,d-1}\rho'_+\cE^{\varphi\circ\mu_{\exp2\pi ik/p}}$. Indeed, this is obtained by writing $\rho^+\rho_+\cE^\varphi=\rho^{\prime+}(\rho_d^+\rho_+\cE^\varphi)$ as the double direct sum of the terms $\cE^{\varphi\circ\mu_{\exp2\pi ik/p}\circ\mu_{\zeta'}}$, where $\zeta'$ varies among the $p'$th roots of the unity and $k$ varies from $0$ to $d-1$.
\end{remarque}

\begin{lemme}\label{lem:isomelem}
We have $\El([u\mto u^p],\varphi,R)\simeq \El([u\mto u^p],\psi,S)$ if and only if the following properties are satisfied:
\begin{enumerate}
\item\label{lem:isomelem1}
there exists $\zeta$ with $\zeta^p=1$ and $\psi\circ\mu_\zeta\equiv\varphi\bmod\Cu$,
\item\label{lem:isomelem2}
$S\simeq R$ as $\Clu$-vector spaces with connection.
\end{enumerate}
\end{lemme}

\begin{proof}
For the ``if'' part, notice that $\rho=\rho\circ\mu_\zeta$. We have $\mu_\zeta^+(\cE^\varphi\otimes R)\simeq\cE^{\varphi\circ\mu_\zeta}\otimes R$, hence $\cE^\varphi\otimes R\simeq\mu_{\zeta+}(\cE^{\varphi\circ\mu_\zeta}\otimes R)$, so $\rho_+(\cE^\varphi\otimes R)\simeq\rho_+(\cE^{\varphi\circ\mu_\zeta}\otimes R)$.

For the ``only if'' part, let us choose a finite dimensional $\Clt$-vector space with regular connection, that we denote by $R^{1/p}$, such that $\rho^+R^{1/p}=R$ (this amounts to choose a $p$th root of an automorphism of a $\CC$-vector space). We then have
$$
\rho_+(\cE^\varphi\otimes R)=\rho_+(\cE^\varphi\otimes \rho^+R^{1/p})=\rho_+\cE^\varphi\otimes R^{1/p},
$$
hence $\rho^+\rho_+(\cE^\varphi\otimes R)=(\rho^+\rho_+\cE^\varphi)\otimes R$. If $\rho_+(\cE^\varphi\otimes R)\simeq\rho_+(\cE^\psi\otimes S)$, we can lift this isomorphism after $\rho^+$. From Lemma \ref{lem:pull-back} and the previous computation we deduce that $\rho^+\rho_+(\cE^\varphi\otimes R)$ decomposes as a sum of terms $\cE^{\varphi\circ\mu_\eta}\otimes R^\nu$, where~$\nu$ is the number of $\zeta$'s such that $\varphi\circ\mu_\zeta\equiv\varphi\bmod\Cu$, and where $\eta$ is such that $\eta^{p/\nu}=1$. It follows that \eqref{lem:isomelem1} is satisfied and that $R^\nu\simeq S^\nu$. Then $R\simeq S$ (if two automorphisms $T,T'$ of a finite dimensional $\CC$-vector space are such that $\oplus_{i=1}^\nu T$ is conjugate to $\oplus_{i=1}^\nu T'$, then $T$ and $T'$ are conjugate: this can be seen by considering the Jordan normal forms).
\end{proof}

\begin{corollaire}\label{cor:isompp}
Let $\El(\rho_1,\varphi_1,R_1)$ and $\El(\rho_2,\varphi_2,R_2)$ be two elementary formal connections with $p_1=p_2=p$. Then $\El(\rho_1,\varphi_1,R_1)\simeq\El(\rho_2,\varphi_2,R_2)$ if and only if there exist~$\zeta$ with $\zeta^p=1$ and $\lambda_1,\lambda_2\in u\Cu$ satisfying $\lambda'_1(0)\neq0$ and $\lambda'_2(0)\neq0$, such that
$$
\rho_1=\rho_2\circ\lambda_1,\quad \varphi_1\equiv\varphi_2\circ\lambda_1\circ(\lambda_2^{-1}\circ\mu_\zeta\circ\lambda_2)\mod\Cu.\eqno\qed
$$
\end{corollaire}

\begin{remarque}\label{rem:pq}
Let $\El(\rho,\varphi,R)$ be an elementary formal connection and assume that there exists a decomposition $\rho=\rho_1\circ\rho_2$ such that $\varphi=\varphi_1\circ\rho_2$. Then
\[
\rho_+(\cE^\varphi\otimes R)= \rho_{1,+}\rho_{2,+}(\cE^{\varphi_1\circ\rho_2}\otimes R)= \rho_{1,+}(\cE^{\varphi_1}\otimes R_1)
\]
where $R_1=\rho_{2,+}R$ has a regular singularity. In other word, $\El(\rho,\varphi,R)=\El(\rho_1,\varphi_1,R_1)$. Hence it is always possible to choose the presentation of an elementary formal connection in a minimal way, such that $\varphi$ cannot be defined on a ramified sub-covering of $\rho$. We will the say that $\rho$ is \emph{minimal} with respect to~$\varphi$.
\end{remarque}

\subsubsection*{Determinant}
We now give a formula for the determinant of $\El(\rho,\varphi,R)$. Recall that, if $M$ is any finite dimensional $\Clt$-module with connection, the determinant $\det M$ is equipped with a natural connection and, if $A$ is the matrix of $\nabla_{\partial_t}$ in some basis of $M$, then the matrix of $\nabla_{\partial_t}$ acting on $\det M$ is $\Tr A$. If the connection on $M$ is regular, then the connection on $\det M$ is completely determined by the residue of its connection modulo $\ZZ$.

\begin{proposition}\label{prop:detel}
The determinant of of the elementary $\Clt$-module with connection $\El(\rho,\varphi,R)$ is isomorphic to $\cE^{r\Tr\varphi}\otimes\det R\otimes (t^{(p-1)r/2})$, where $p$ is the degree of~$\rho$, $r$ is the rank of $R$, $(t^{(p-1)r/2})$ is the rank one free $\Clt$-module with connection \hbox{$d+[(p-1)r/2]dt/t$}, and $\det R$ is the rank one $\Clu$-module $\det R$ where we change the name of the variable~$u$ to~$t$.
\end{proposition}

\begin{proof}
As $R$ is a successive extension of rank-one free $\Clu$-modules with regular connection, we can reduce to the case $R$ has rank one. We can also assume that $\rho(u)=u^p$. If $\varphi\in\Clu$, we denote by $\varphi^\inv$ its invariant part with respect to the action of $\ZZ/p\ZZ$.

If $e$ is a $\Clu$-basis of the rank-one module $\cE^\varphi\otimes R$, then $e,ue,\dots,u^{p-1}e$ is a $\Clt$-basis of $\rho_*(\cE^\varphi\otimes R)$. If $\alpha\in\CC$ is the residue of $(R,\nabla)$, then the matrix of the action of $t\partial_t$ in this basis is given by
\[
\frac1p\big[\alpha\Id+\diag(0,\dots,p-1)+u\varphi'_u\cdot{}\big]
\]
where the multiplication by $u\varphi'_u$ has to be interpreted as an operator on the $\Clt$-module $\Clu$. We note that multiplication by $u^\ell$ on $\Clu$ has trace zero except when $\ell$ is a multiple of $p$. Thus $\frac1p\Tr[u\varphi'_u\cdot{}]=\frac1p\Tr([u\varphi'_u\cdot{}]^\inv)= t(\Tr\varphi)'_t$. Therefore, the trace of the matrix of the action of $t\partial_t$ on the basis $e,\dots,u^{p-1}e$ is
\[
\alpha+\frac{p-1}2+t(\Tr\varphi)'_t.\qedhere
\]
\end{proof}

\begin{corollaire}\label{cor:detreg}
If $\textup{slope}\El(\rho,\varphi,R)<1$, then $\det \El(\rho,\varphi,R)$ has a regular singularity.
\end{corollaire}

\begin{proof}
Indeed, $q<p$ implies $\Tr\varphi=0\bmod\Ct$.
\end{proof}

\section{Formal decomposition of a germ of meromorphic connection}\label{sec:turrittin}

\begin{proposition}\label{prop:irred}
Any irreducible finite dimensional $\Clt$-vector space $M$ with connection is isomorphic to $\rho_+(\cE^\varphi\otimes L)$, where $\varphi\in\um\CC[\um]$, $\rho:u\mto t=u^p$ has degree $p\geq1$ and is minimal with respect to $\varphi$ (\cf Remark \ref{rem:pq}), and $L$ is some rank one $\Clu$-vector space with regular connection.
\end{proposition}

\begin{proof}
Assume that $M$ is irreducible. According to the classical Turrittin-Levelt theorem, we can choose $\rho:u\mto t=u^p$ such that $\rho^+M$ decomposes as $\oplus_\varphi(\cE^\varphi\otimes R_\varphi)$. We have a natural action of $\ZZ/p\ZZ$ on $\rho^+M$, and~$M$ is recovered as the invariant subspace. Therefore, by irreducibility, there exists $\varphi\in\Clu$ such that $\rho^+M\simeq\oplus_{\zeta^p=1}(\cE^{\varphi\circ\mu_\zeta}\otimes R_{\varphi\circ\mu_\zeta})$ and we can assume that $p$ is minimal, so that for any $\zeta\neq1$ with $\zeta^p=1$, we have $\varphi\circ\mu_\zeta\not\equiv\varphi\bmod\Cu$. The $\ZZ/p\ZZ$-action induces isomorphisms $a_\zeta:R_\varphi\to R_{\varphi\circ\mu_\zeta}$ which compose themselves in the right way. Let us set $R=R_\varphi$ and let us choose $R^{1/p}$ as in the proof of Lemma \ref{lem:isomelem}. Then $\rho^+M\simeq\rho^+(\rho_+\cE^\varphi\otimes R^{1/p})$ and, taking the $\ZZ/p\ZZ$-invariant part, we get $M\simeq\rho_+\cE^\varphi\otimes R^{1/p}$. Still by irreducibility, $R^{1/p}$, hence $R$, has rank one.

On the other hand, such a $\rho_+(\cE^\varphi\otimes L)$ is irreducible: indeed, $\rho^+\rho_+(\cE^\varphi\otimes L)$ decomposes as the direct sum of rank-one non-isomorphic connections, and has no non-trivial $\ZZ/p\ZZ$-invariant submodule.
\end{proof}

\begin{remarque}
There is no unique way, in general, to write down an irreducible meromorphic connection: either we choose the presentation $\rho_+(\cE^\varphi\otimes_{\Clu} L)$, and $L$ is uniquely defined up to isomorphism, but $\varphi$ could be changed into $\varphi\circ\mu_\zeta$, or we choose $\rho_+(\cE^\varphi)\otimes_{\Clt} L^{1/p}$, and $L^{1/p}$ is not uniquely defined. We will call $\rho_+\cE^\varphi$ the exponential irreducibility type of the irreducible meromorphic connection.
\end{remarque}

\begin{corollaire}[Refined Turrittin-Levelt, \cf \cite{B-E04}]\label{cor:refinedTL}
Any finite dimensional $\Clt$-vector space $M$ with connection can be written in a unique way as a direct sum $\bigoplus\El(\rho,\varphi,R)$, in such a way that each $\rho_+\cE^\varphi$ is irreducible and no two $\rho_+\cE^\varphi$ are isomorphic.
\end{corollaire}

\begin{proof}
Fix an irreducible $\Clt$-module $I$ with connection and consider in~$M$ the $I$-typical component $M_I$ defined as the maximal submodule such that all irreducible sub-quotients are isomorphic to $I\otimes L$ for some rank one regular $\Clt$-module with connection~$L$. It will be convenient to choose $I$ as $\rho_+\cE^\varphi$ for some suitable~$\rho$ and~$\varphi$ as in Proposition \ref{prop:irred}. Then, if $I_1\not\simeq I_2$, there is no non-zero morphism from $M_{I_1}$ to $M_{I_2}$, and we get the decomposition $M=\oplus_IM_I$. On the other hand, each $M_I$ is isomorphic to $\rho_+(\cE^\varphi\otimes R)$ for some regular $R$.
\end{proof}

\begin{corollaire}\label{cor:Evphi}
Let $M$ be a finite dimensional $\Clt$-vector space with connection. Then $M$ is isomorphic to an elementary module $\El(\rho,\varphi,R)$ if and only if $\rho^+M$ is isomorphic to $\rho^+\El(\rho,\varphi,R)=(\rho^+\rho_+\cE^\varphi)\otimes R$.\qed
\end{corollaire}

\begin{corollaire}\label{cor:isommin}
Let $\El(\rho_1,\varphi_1,R_1)$ and $\El(\rho_2,\varphi_2,R_2)$ be two elementary connections written in a minimal way (\cf Remark \ref{rem:pq}). Then $\El(\rho_1,\varphi_1,R_1)\simeq\El(\rho_2,\varphi_2,R_2)$ if and only if $p_1=p_2$ and the condition of Corollary~\ref{cor:isompp} applies.
\end{corollaire}

\begin{proof}
Any irreducible sub-quotient of $\El(\rho_1,\varphi_1,R_1)$ takes the form $\El(\rho_1,\varphi_1,L_1)$ where $L_1$ is a rank one sub-quotient of $R_1$. A similar result holds for $\El(\rho_2,\varphi_2,R_2)$. Each such sub-quotient has rank $p_1$ (\resp $p_2$). It follows that $p_1=p_2$. We can then apply Corollary~\ref{cor:isompp}.
\end{proof}

\begin{remarque}[Extension to $\wh\cD$-modules]\label{rem:reg}
Let us denote by $\wh\cD$ the ring of differential operators with coefficients in $\Ct$. Then any given holonomic $\wh\cD$-module $M$ can be decomposed as $M_\reg\oplus M_\irr$, where $M_\reg$ is a regular holonomic $\wh\cD$-module and $M_\irr$ is purely irregular. Then $M_\irr$ is a $\Clt$-module with connection (although $M_\reg$ may be not), and has the refined Turrittin-Levelt decomposition of Corollary \ref{cor:refinedTL}, where all~$\varphi$ are non-zero in $\Clu/\Cu$.
\end{remarque}

\subsubsection*{Tensor product}
Given two elementary formal connections $\El([u\mto u^{p_1}],\varphi_1,R_1)$ and $\El([v\mto v^{p_2}],\varphi_2,R_2)$, we set $d=\gcd(p_1,p_2)$, $p'_1=p_1/d$ and $p'_2=p_2/d$. We have the following diagram:
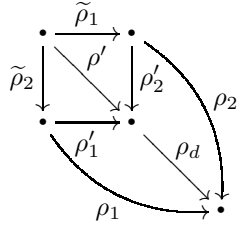
\[
\xymatrix{
\cbbullet\ar[r]^-{\wt \rho_1}\ar[d]_{\wt\rho_2}\ar[dr]^-{\!\!\rho'}& \cbbullet\ar[d]^{\rho'_2}\ar@/^1pc/[rdd]^-{\rho_2}\\
\cbbullet\ar[r]_{\rho'_1}\ar@/_1pc/[rrd]_-{\rho_1}&\cbbullet\ar[rd]^{\!\!\rho_d}\\
&&\cbbullet
}
\]
where the dots represent formal discs, $\rho'_1(u)=u^{p'_1}$, $\wt\rho_1(w)=w^{p'_1}$, \etc We then set
\begin{equation}\label{eq:formulestens}
\begin{split}
\rho(w)&=w^{p_1p_2/d},\\
\varphi^{(k)}&=\varphi_1(w^{p'_2})+\varphi_2([e^{2\pi ikd/p_1p_2}w]^{p'_1})\quad (k=0,\dots,d-1),\\
R&=\wt\rho_2^+R_1\otimes\wt\rho_1^+R_2.
\end{split}
\end{equation}

\begin{proposition}\label{prop:tens}
With this notation,
\[
\El([u\mapsto u^{p_1}],\varphi_1,R_1)\otimes \El([v\mapsto v^{p_2}],\varphi_2,R_2)\simeq\tbigoplus_{k=0}^d\El(\rho,\varphi^{(k)},R).
\]
\end{proposition}

\begin{proof}
As $\rho^+(R_1^{1/p_1}\otimes R_2^{1/p_2})\simeq \wt\rho_2^+R_1\otimes\wt\rho_1^+R_2$, we are reduced to finding an isomorphism
\[
\rho_{1,+}\cE^{\varphi_1}\otimes\rho_{2,+}\cE^{\varphi_2}\simeq\tbigoplus_{k=0}^{d-1}\rho_+\cE^{\varphi^{(k)}}.
\]

If $d=1$ (so that $p'_1=p_1$ and $p'_2=p_2$), we have
\[
\rho^+(\rho_{1,+}\cE^{\varphi_1}\otimes\rho_{2,+}\cE^{\varphi_2})=
\hspace*{-1mm}
\tbigoplus_{\substack{\zeta_1^{p_1}=1\\\zeta_2^{p_2}=1}}
\hspace*{-1mm}
\cE^{\varphi_1(\zeta_1w^{p_2})+\varphi_2(\zeta_2w^{p_1})}=
\hspace*{-1mm}
\tbigoplus_{\substack{\zeta_1^{p_1}=1\\\zeta_2^{p_2}=1}}
\hspace*{-1mm}
\cE^{\varphi(\zeta_1\zeta_2w)}=
\hspace*{-1mm}
\tbigoplus_{\zeta^{p_1p_2}=1}
\hspace*{-2mm}
\cE^{\varphi\circ\mu_\zeta},
\]
and we get the desired isomorphism by taking the $\ZZ/p_1p_2\ZZ$-invariant part.

Otherwise, we have $\rho_{1,+}\cE^{\varphi_1}\otimes\rho_{2,+}\cE^{\varphi_2}\simeq\rho_{d,+}(\rho'_{1,+}\cE^{\varphi_1}\otimes\rho_d^+\rho_{2,+}\cE^{\varphi_2})$ and, using Remark \ref{rem:partialdecomp}, this is $\bigoplus_{k=0}^{d-1}\rho_{d,+}(\rho'_{1,+}\cE^{\varphi_1}\otimes\rho'_{2,+}\cE^{\varphi_2\circ\mu_{\exp2\pi ik/p_2}})$. We deduce then from the $d=1$-case that 
\[
\rho'_{1,+}\cE^{\varphi_1}\otimes\rho_d^+\rho'_{2,+}\cE^{\varphi_2}\simeq\tbigoplus_{k=0}^{d-1}\rho'_+\cE^{\varphi^{(k)}},
\]
hence the result by applying $\rho_{d,+}$.
\end{proof}

\begin{remarque}\label{rem:tensor}
Using the notation $(p,q,r)$ as after Definition \ref{def:elem}, we get
\begin{align*}
p&=p_1p_2/\gcd(p_1,p_2),\\
q^{(k)}&\leq \max\{q_1p_2/\gcd(p_1,p_2),q_2p_1/\gcd(p_1,p_2)\},\\
r&=r_1r_2.
\end{align*}
Therefore,
\begin{align*}
\text{slope}\El(\rho,\varphi^{(k)},R)&\leq\max\{\text{slope}\El(\rho_1,\varphi_1,R_1),\text{slope}\El(\rho_2,\varphi_2,R_2)\},
\\
\irr_0\El(\rho,\varphi^{(k)},R)&\leq\max\big\{\irr_0\El(\rho_1,\varphi_1,R_1)\cdot\rg\El(\rho_2,\varphi_2,R_2),\\
&\hspace*{3cm}\irr_0\El(\rho_2,\varphi_2,R_2)\cdot\rg\El(\rho_1,\varphi_1,R_1)\big\},
\\
\rg\El(\rho,\varphi^{(k)},R)&=\rg\El(\rho_1,\varphi_1,R_1)\cdot\rg\El(\rho_2,\varphi_2,R_2)/\gcd(p_1,p_2).
\end{align*}
\end{remarque}

\subsubsection*{Dual}
Using that $\rho_+(N^*)\simeq(\rho_+N)^*$, we get
\begin{equation}\label{eq:dual}
\El(\rho,\varphi,R)^*\simeq\El(\rho,-\varphi,R^*).
\end{equation}

\subsubsection*{Hom}
As a consequence of Proposition \ref{prop:tens} and \eqref{eq:dual} we get:
\begin{equation}\label{eq:Hom}
\begin{split}
\Hom_{\Clt}\big(\El([u\mapsto u^{p_1}],&\varphi_1,R_1),\El([v\mapsto v^{p_2}],\varphi_2,R_2)\big)\\
&\simeq\El([u\mapsto u^{p_1}],-\varphi_1,R_1^*)\otimes\El([v\mapsto v^{p_2}],\varphi_2,R_2)\\
&\simeq\tbigoplus_{k=0}^{d-1}\El([w\mapsto w^{p_1p_2/d}],\varphi^{(k)},R)
\end{split}
\end{equation}
with
\begin{equation}\label{eq:Homsetting}
\varphi^{(k)}(w)=\varphi_2(w^{p'_1})-\varphi_1([e^{2\pi ikd/p_1p_2}w]^{p'_2}),\quad R=\wt\rho_2^+R_1^*\otimes\wt\rho_1^+R_2.
\end{equation}

Applying this formula to $\End_{\Clt}(\El[u\mto u^p],\varphi,R)$ gives
\[
\End_{\Clt}(\El([u\mapsto u^p],\varphi,R)\simeq\tbigoplus_{\zeta^p=1}\El([u\mapsto u^p],\varphi-\varphi\circ\mu_\zeta,\End_{\Clu}(R)).
\]
If we assume that $p$ is minimal with respect to $\varphi$, \ie $\varphi\not\equiv\varphi\circ\mu_\zeta\bmod\Cu$ if $\zeta\neq1$, then we obtain
\[
\End_{\Clt}\big(\El(\rho,\varphi,R)\big)_\reg\simeq\rho_+\End_{\Clu}(R).
\]

On the other hand, in \eqref{eq:Hom}, let us assume that $p_i$ is minimal with respect to~$\varphi_i$ ($i=1,2$) and that $\rho_{1,+}\cE^{\varphi_1}\not\simeq\rho_{2,+}\cE^{\varphi_2}$. We will then show that
\[
\Hom_{\Clt}\big(\El(\rho_1,\varphi_1,R_1),\El(\rho_2,\varphi_2,R_2)\big)_\reg=0.
\]
Indeed, we can assume that $\rho_1(u)=u^{p_1}$ and $\rho_2(v)=v^{p_2}$. Moreover, it is enough to prove the assertion when $R_1=\bun$ and $R_2=\bun$. In such a case, if $\Hom_{\Clt}(\rho_{1,+}\cE^{\varphi_1},\rho_{2,+}\cE^{\varphi_2})_\reg\neq0$, \eqref{eq:Hom} would imply that it is isomorphic to a direct sum of terms $\rho_+\bun$. In particular, it would contain a horizontal section, in contradiction with the assumption.

Let now $M$ be a $\Clt$-vector space with connection, let us denote by $\bigoplus_i\El(\rho_i,\varphi_i,R_i)$ its refined Turrittin-Levelt decomposition, as in Corollary \ref{cor:refinedTL} and let us assume that the minimality condition holds for each $(\rho_i,\varphi_i)$. We then conclude:
\begin{equation}
\label{eq:endreg}
(\End_{\Clt}M)_\reg\simeq\tbigoplus_i\rho_{i,+}\End(R_i).
\end{equation}

\begin{remarque}[Centralizers]\label{rem:Zrho}
Let $R$ be a regular $\Clu$-module corresponding to a vector space with monodromy $(\psi_uR,T)$. If $\rho\in u\Cu$ has valuation equal to $p\geq1$, then $\rho_+R$ corresponds to $(\psi_uR\otimes\CC^p,\rho_+T)$, where $\rho_+T=T^{1/p}\otimes \rP_p$ for some choice of a $p$th root of $T$ and where $\rP_p$ is the cyclic permutation matrix on $\CC^p$. The following is easy:
\[
\dim \ker(\rho_+T-\Id)=p\dim \ker(T-\Id).
\]
Applying this equality to $\End_{\Clu}(R)$ and denoting by $\rZ(T)$ the centralizer of $T$ ($=\ker(\Ad(T)-\Id)$), we obtain
\[
\dim\ker(\rho_+\Ad(T)-\Id)=p\dim \ker(\Ad(T)-\Id)=p\dim\rZ(T).
\]
\end{remarque}

\section{Direct images of exponentially twisted regular $\cD$-modules}\label{sec:exptwist}

In this section we explain an improvement, in a particular case, of the main result of \cite{Roucairol07}. Let $\Delta$ be a disc centered at the origin in $\CC$ with coordinate $t$ and let~$\PP^1$ be the projective line with affine chart $\Afu$ having coordinate $x$. We denote by~$\infty$ the point with coordinate $x=\infty$. Let $\cM$ be a \emph{regular} holonomic $\cD_{\Delta\times\PP^1}$-module. In the following, we always assume that $\cM$ is equal to its localized module along the divisor $[\Delta\times\{\infty\}]\cup[\{0\}\times\PP^1]$ (the adjunction of this second component will not affect the computation of the irregularity we are interested to compute). Let \hbox{$t:\Delta\times\PP^1\to \Delta$} be the projection and let $\cE^x$ denotes the free rank-one $\cO_{\Delta\times\PP^1}(*\infty)$-module with connection $d+dx$. The main result of \cite{Roucairol07} gives much information on the formal irregular part $\cH^0t_+(\cE^x\otimes\cM)^\irr$ of $\cH^0t_+(\cE^x\otimes\cM)$ at the origin of $\Delta$. Let us recall the notation of \loccit.

The singular support of $\cM$ away from $\{0\}\times\PP^1$ (\ie the locus where $\cM$ is not locally a vector bundle with flat connection away from $\{0\}\times\PP^1$) is a germ along $\{0\}\times\PP^1$ of a analytic closed analytic set of $\Delta\times\PP^1$ of dimension $\leq1$. Therefore, if~$\Delta$ is small enough, it is a union of germs, at a finite number of points of $\{0\}\times\PP^1$, of possibly singular complex analytic curves (distinct from $\{0\}\times\PP^1$). We will denote by $S$ the germ at $(0,\infty)\in \Delta\times\PP^1$ of the singular support of $\cM$ away from $\{0\}\times\PP^1$. Let us denote by $y=1/x$ the coordinate at $\infty$ on $\PP^1$. We first make the following assumption:

\begin{assumption}\label{assumpt:irred}
The germ of $S$ at $(0,\infty)$ is irreducible.
\end{assumption}

We fix a Puiseux parametrization of $S$ as $t=u^{p}$, $\mu(u)y=u^{q}$ where $\mu$ is holomorphic and $\mu(0)\neq0$. We assume that $p$ is minimal, in the sense that one cannot find a Puiseux parametrization of $S$ with a smaller $p$. The inverse image $\rho^{-1}S$ of $S$ by \hbox{$\rho:(u,y)\mto (u^p,y)$} consists of $p$ distinct smooth curves $S_\zeta$ having equation $\zeta^{-q}\mu(\zeta u)y=u^q$, where $\zeta$ varies among the $p$th roots of the unity. In particular, the restriction of~$\rho$ to any of the curves $S_\zeta$ is isomorphic to the normalization $\nu:\wt S\to S$.

We define $\alpha$ to be the polar part of $u^{-q}\mu(u)$, and $\delta$ to be its holomorphic part. If $\zeta$ varies among the $p$th roots of the unity, the functions $\zeta^{-q}\mu(\zeta u)$ are thus pairwise distinct. We will make the more restrictive following assumption, which is satisfied if $(p,q)=1$:

\begin{assumption}\label{assumpt:irred2}
The polar parts $\alpha(\zeta u)$ are pairwise distinct when $\zeta$ varies in the set of $p$th roots of the unity.
\end{assumption}

Let us consider the de~Rham complex $\DR\cM$. Fixing a local equation $h$ of $S$ at $(0,\infty)$, we can define a constructible complex on $S$ by taking the vanishing cycle complex $\phi_{h}\DR\cM$, whose natural monodromy around $h=0$ we forget. Up to a shift (depending on the convention we took for perverse sheaves and vanishing cycle functor), the restriction of this complex to $S^*\defin S\moins\{(0,\infty)\}$ is a local system. Let $\nu:\wt S\to S$ be the normalization. The previous local system defines thus a vector bundle $R$ with meromorphic connection having a regular singularity at the origin of~$\wt S$.

\begin{theoreme}\label{th:rouc}
If $S$ is irreducible at $(0,\infty)$ and if Assumption \ref{assumpt:irred2} is satisfied, the formal irregular part $\cH^0t_+(\cE^x\otimes\cM)^\irr$ at the origin of~$\Delta$ is isomorphic to $\El(t\circ\nu,x\circ\nu,R)$.
\end{theoreme}

In \cite{Roucairol07}, C\ptbl Roucairol obtains a similar result, but the regular part $R$ is not computed so precisely, as only the characteristic polynomial of the monodromy is determined, not its Jordan structure. We will show below how the argument of \loccit.\ can be modified to get the more precise result of Theorem \ref{th:rouc}.

\begin{proof}
Let us start with the following situation. We consider a regular holonomic $\cD_{\Delta\times\PP^1}$-module (localized along $[\Delta\times\{\infty\}]\cup[\{0\}\times\PP^1]$ as above) with singular support consisting, near $(0,\infty)$, of a finite family of smooth curves $S_i$ having local equations $\mu_i(t)y=t^{q_i}$, with $q_i\geq1$, $\mu_i$ holomorphic and $\mu_i(0)\neq0$.
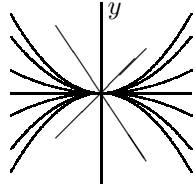
\begin{center}
\setlength{\unitlength}{.3mm}
\begin{picture}(80,80)(0,0)
\qbezier(0,75)(40,5)(80,75)
\qbezier(0,65)(40,15)(80,65)
\qbezier(0,50)(40,30)(80,50)
\qbezier(0,5)(40,75)(80,5)
\qbezier(0,15)(40,65)(80,15)
\qbezier(0,30)(40,50)(80,30)
\put(20,20){\line(1,1){40}}
\put(20,70){\line(2,-3){40}}
\thicklines
\put(0,40){\line(1,0){80}}
\put(40,0){\line(0,1){80}}
\put(43,75){$y$}
\end{picture}
\end{center}

We denote by $\alpha_i\in\tm\CC[\tm]$ the polar part of $t^{-q_i}\mu_i$ and by $\delta_i$ its holomorphic part. We define the meromorphic bundle with connection $R_i$ on $S_i$ by using the same procedure as above. We thus get a vector space with an automorphism $(\psi_tR_i,T_i)$.

Let $\alpha$ be any nonzero element of $\tm\CC[\tm]$ and let $q$ the order of its pole, so that, if we set $\mu(t)=t^q\alpha(t)$, we have $\mu(0)\neq0$. Let us note that $\psi_t(\cM\otimes\nobreak\cE^{x-\alpha})$ is supported at $(0,\infty)$ (near $x^o$ at finite distance, $\cM\otimes\cE^{x-\alpha})\simeq\cM\otimes\cE^{-\alpha})$ and the local $V$-filtration is easily seen to be constant). Therefore, $\psi_t(\cM\otimes\cE^{x-\alpha})$ being holonomic (see \hbox{\eg \cite{M-S86b}}) and supported on a point, it is equivalent (through Kashiwara's equivalence) to a finite dimensional vector space with an automorphism~$T$. We can analyze it by working in the~$y$ coordinate only. We now characterize this object:

\begin{proposition}[Improvement of \cite{Roucairol07}]\label{prop:rouc}
In such a situation, let us moreover assume that the functions $\alpha_i(t)+\delta_i(0)$ are pairwise distinct. Then, for any $\alpha\in\tm\CC[\tm]$,
\[
(\psi_t(\cM\otimes\cE^{-\alpha+1/y}),T)\simeq\tbigoplus_{i\mid\alpha_i=\alpha}(\psi_tR_i,T_i).
\]
\end{proposition}

\begin{proof}
Let $\pi_1$ denote the $q$-times composition of the blowing up of the intersection of the successive strict transforms of the curve $\mu(t)y=t^q$ with the exceptional divisor, starting with the point $(0,\infty)$. Computing as in \cite{Roucairol07}, we find that \hbox{$\psi_{t\circ\pi_1}\pi_1^+(\cM\otimes\cE^{-\alpha+1/y})$} is supported at the intersection $P_1$ of the strict transform of the curve $\mu(t)y=t^q$ with the exceptional divisor. There is a chart $(u,v)$ such that $t\circ\pi_1=v$ and $y\circ\pi_1=uv^q$. Let us set $u_1=1-u\mu(v)$, which is a coordinate centered at $P_1$. In this chart, the strict transform of $\mu_i(t)y=t^{q_i}$ has equation $\mu_i(v)u=v^{q_i-q}$ and goes through $P_1$ if and only if $q_i=q$ and $\mu_i(0)=\mu(0)$. In the coordinate system $(u_1,v)$, the equation of such a curve is $u_1=1-\mu(v)/\mu_i(v)$. On the other hand, $\pi_1^+(\cM\otimes\cE^{-\alpha+1/y})=\pi_1^+\cM\otimes\cE^{u_1/v^q}$.

Let $\pi_2$ denote the $q$-times composition of the blowing up of the intersection of the successive strict transforms of $u_1=0$ with the exceptional divisor. Let $E_{\text{last}}\simeq\PP^1$ be the last created component of the exceptional divisor $E$. It cuts $\ov{E\moins E_{\text{last}}}$ at a unique point $P$. There is a chart with coordinates $(w,v)$ such that $v=0$ is the equation of $E_{\text{last}}$, $w$ is an affine coordinate on $E_{\text{last}}$, and $P=\{w=\infty\}$. In this chart, we have $u_1\circ\pi_2=wv^q$ and $v\circ\pi_2=v$.

The strict transform of a curve $u_1=1-\mu(v)/\mu_i(v)$ has equation
\[
w=v^{-q}\,\frac{\mu_i(v)-\mu(v)}{\mu_i(v)}=\frac{(\alpha_i(v)+\delta_i(0)-\alpha(v)}{\mu_i(v)}+v\nu_i(v),
\]
with $\nu_i$ holomorphic.
It cuts $E_{\text{last}}$ away from $P$ if and only if $\alpha_i=\alpha$, and the intersection point is located at $w=\delta_i(0)/\mu_i(0)=\delta_i(0)/\mu(0)$. As a consequence, along $E_{\text{last}}\moins\{P\}$, the pull-back of $\bigcup_iS_i$ is a normally crossing divisor.

On the other hand, in this chart, we have $\pi_2^+\pi_1^+(\cM\otimes\cE^{-\alpha+1/y})=\pi_2^+\pi_1^+(\cM)\otimes\nobreak\cE^w$ and $\psi_t[\pi_2^+\pi_1^+(\cM)\otimes\cE^w]=\psi_t[\pi_2^+\pi_1^+\cM]\otimes\cE^w$.

Let us analyze what happens near $P$. In the neighbourhood of $P$, we have coordinates $(v',w')$, where $v'=0$ is the equation of $E_{\text{last}}$ and $w'=1/w$. It may happen that the strict transform of some $S_i$ goes through $P$. However, we can find a sequence of blowing-ups over $P$ such that, if we still denote by $P$ the point at infinity in the strict transform of $E_{\text{last}}$ by this sequence, then none of these strict transforms goes through $P$. Let $\pi$ denote the composition of this sequence with $\pi_1$ and $\pi_2$. In the chart centered at $P$, we have coordinates $(v'',w'')$ such that $v''=0$ is the equation of $E_{\text{last}}$ and, on $E_{\text{last}}$, $w''=1/w$. Near $P$, we have $\psi_{t\circ\pi}\pi^+(\cM\otimes\nobreak\cE^{-\alpha+1/y})=\psi_{v^{\prime\prime k}w^{\prime\prime\ell}}(\pi^+\cM\otimes\cE^{1/w''})$ with $k,\ell\in\NN$. Applying Proposition \ref{prop:bibi}\eqref{prop:bibi2}, we find that $\psi_{t\circ\pi}\pi^+(\cM\otimes\cE^{-\alpha+1/y})=\big(\psi_{t\circ\pi}\pi^+\cM\big)[w^{\prime\prime-1}]\otimes\nobreak\cE^{1/w''}$.

Let us denote by $\cN$ the regular holonomic module $\psi_{t\circ\pi}\pi^+\cM$ restricted to $E_{\text{last}}$ and localized at $P$. We identify $E_{\text{last}}$ with $\PP^1$ with its affine coordinate~$w$ and~$P$ with~\hbox{$w=\infty$}. Then $N=\Gamma(E_{\text{last}},\cN)$ is a regular holonomic $\CC[w]\langle\partial_w\rangle$-module equipped with an automorphism $T$ (coming with the functor $\psi_{t\circ\pi}$). Summing up, $(\psi_t(\cM\otimes\cE^{-\alpha+1/y}),T)$ is the direct image by the constant map $\pi:E_{\text{last}}\to(0,\infty)$ of $\cN\otimes\cE^w$, which can be computed as the cokernel $\wh N_1$ of the map $\partial_w+1:N\to N$, equipped with the induced automorphism~$\wh T_1$. Now, Proposition \ref{prop:rouc} is a consequence of the following lemma:

\begin{lemme}\label{lem:malg}
Let $N$ be a regular holonomic $\CC[w]\langle\partial_w\rangle$-module equipped with an automorphism $T$. Let $\wh N_1$ denote the cokernel of the (injective) map $\partial_w+\nobreak1:N\to N$, equipped with the induced automorphism $\wh T_1$. Then we have $(\wh N_1,\wh T_1)\simeq\oplus_{c\in\CC}(\phi_{w-c}N,\phi_{w-c}T)$.
\end{lemme}

Indeed, if the lemma is proved, we are reduced to computing $(\phi_{w-c}N,\phi_{w-c}T)$. Because the singular support of $\pi^+\cM$ has normal crossings along $E_{\text{last}}$, we have
\[
(\phi_{w-c}\psi_{t\circ\pi}\pi^+\cM,\phi_{w-c}T)=(\psi_{t\circ\pi}\phi_{w-c}\pi^+\cM,T_c)=(\psi_tR_i,T_i)
\]
if $c=\delta_i(0)$.
\end{proof}

\begin{proof}[Proof of Lemma \ref{lem:malg}]
The lemma is well-known if we forget $T$ (see \eg \cite[Prop\ptbl 1.5, p\ptbl 79]{Malgrange91}). Let us show how to take $T$ into account. We note that $T$ has a minimal polynomial, so we can decompose $N$ with respect to eigenvalues of $T$ and we are reduced to the case where $T$ is unipotent. We have then to prove that $\wh T_1$ and $\oplus_c\phi_{w-c}T$ have the same Jordan normal form. Let $\rM_\bbullet N$ denote the monodromy filtration of the nilpotent endomorphism $\log T$. It induces a filtration $\rM_\bbullet\wh N_1$ and $\oplus_c\rM_\bbullet\phi_{w-c}N$ and, as $\coker(\partial_w+1)$ and $\phi_{w-c}$ are exact functors on regular holonomic modules, we obtain, by applying the lemma forgetting automorphisms,
\[
\forall k\in\ZZ,\quad \gr_k^\rM\wh N_1=\wh{\gr_k^\rM N}_1\simeq\oplus_c\phi_{w-c}\gr_k^\rM N=\oplus_c\gr_k^\rM\phi_{w-c}N.
\]
If we prove that $\rM_\bbullet\wh N_1$ (\resp $\rM_\bbullet\phi_{w-c}N$) is the monodromy filtration of $\log\wh T_1$ (\resp $\log\phi_{w-c}T$), then $\log\wh T_1$ and $\oplus_c\log\phi_{w-c}T$ will have the same number of Jordan blocks of any given size, hence will have the same Jordan normal form.

Recall that the monodromy filtration $\rM_\bbullet(\ccN)$ of a nilpotent endomorphism $\ccN$ is characterized by two properties: (1) for any $\ell\in\ZZ$, $\ccN(\rM_\ell(\ccN))\subset\rM_{\ell-1}(\ccN)$, and (2)~for any $\ell\in\NN$, $\ccN^\ell$ induces an isomorphism $\gr_\ell^{\rM(\ccN)}\isom\gr_{-\ell}^{\rM(\ccN)}$. Using the exactness of the functors $\coker(\partial_w+1)$ and $\phi_{w-c}$, one easily checks that $\rM_\bbullet\wh N_1$ (\resp $\rM_\bbullet\phi_{w-c}N$) satisfies these characteristic properties for $\log\wh T_1$ (\resp for $\log\phi_{w-c}T$).
\end{proof}

\subsubsection*{End of the proof of Theorem~\ref{th:rouc}}
We assume that $S$ is non-empty, otherwise we know from \cite{Roucairol07} that $\cH^0t_+(\cE^x\otimes\cM)$ has a regular singularity at $t=0$. After Corollary~\ref{cor:Evphi}, it is enough to compute $\rho^+\cH^0t_+(\cE^x\otimes\cM)$ where $\rho$ is defined after Assumption~\ref{assumpt:irred}. Moreover, a standard argument enables us to apply $\rho$ before $\cH^0t_+$. The singular support of $\rho^+\cM$ near $(0,\infty)$ is $\rho^{-1}S=\bigcup_{\zeta^p=1}S_\zeta$. We will use Proposition~\ref{prop:rouc} with $\alpha=\alpha_1$. Assumption \ref{assumpt:irred2} implies that the assumption in Proposition \ref{prop:rouc} is fulfilled and moreover, $S_1$ is the only component which cuts $E_{\text{last}}$. As $\rho:S_1\to S$ is the normalization of $S$, $(\psi_tR_1,T_1)$ is identified with $(\psi_tR,T)$.
\end{proof}

\section{Local Laplace transform}\label{sec:lapalceloc}

\subsection{The local Laplace transform $\cF^{(0,\infty)}$}
In this section, we analyze the formal Laplace transform $\cF^{(0,\infty)}_\pm$ introduced in \cite{B-E04,Garcia04}. Let us recall that, given a finite dimensional $\Clt$-vector space $M$ with connection, its transform $\cF^{(0,\infty)}_\pm M$ is a finite dimensional $\Clth$-vector space with connection, where $\theta$ is a new variable. The functor $\cF^{(0,\infty)}_\pm$ morally corresponds to the integral transform $\int \cbbullet\, e^{\pm t/\theta}dt$. If $\iota$ denotes the germ of formal automorphism $t\mto-t$, we have $\cF^{(0,\infty)}_+M=\cF^{(0,\infty)}_-\iota^+M$. Let us note that, if $R$ is regular, then $\iota^+R\simeq R$.

\begin{theoreme}\label{th:main}
For any elementary $\Clt$-vector space $\El(\rho,\varphi,R)$ with irregular connection (\ie such that $\varphi\not\in\Cu$), the formal Laplace transform $\cF^{(0,\infty)}_\pm\El(\rho,\varphi,R)$ is the elementary finite dimensional $\Clth$-vector space with connection $\El(\wh\rho_\pm,\wh\varphi,\wh R)$ with (setting $L_q=(\Clu,d-\tfrac q2\,\tfrac{du}{u})$)
\[
\wh\rho_\pm(u)=\mp\frac{\rho'(u)}{\varphi'(u)},\quad
\wh\varphi(u)=\varphi(u)-\frac{\rho(u)}{\rho'(u)}\,\varphi'(u),\quad
\wh R\simeq R\otimes L_q.
\]
\end{theoreme}

\begin{remarque}[Regular connections]
If $R$ is a regular connection, then it is easy to check that $\cF^{(0,\infty)}_\pm R\simeq R$ (after replacing the name of the variable~$t$ with~$\theta$). This also explains why $\wh R$ does not depend on $\pm$ in the formula above.
\end{remarque}

We will prove the theorem for the transform $\cF^{(0,\infty)}_-$. The other case is obtained by setting $\cF^{(0,\infty)}_+=\cF^{(0,\infty)}_-\circ\iota^+$. In the following, when we forget the index $\pm$, we implicitly work in the $-$ case.

\begin{remarques}\label{rem:main}\mbox{}
\begin{enumerate}
\item\label{rem:main1}
Using notation of Remark \ref{rem:pq}, one has $\wh\rho=\wh{\rho_1}\circ\rho_2$, $\wh\varphi=\wh{\varphi_1}\circ\rho_2$ and $\wh{R_1}=R_1\otimes L_{q_1}=\rho_{2,+}R\otimes L_{q_1}=\rho_{2,+}(R\otimes L_q)=\rho_{2,+}\wh R$. Therefore, the formulas above do not depend on whether the presentation of $\El(\rho,\varphi,R)$ is minimal or not and we can also assume that $\rho(u)=u^p$.

\item\label{rem:main2}
Let us note that $\wh\rho$ is a ramification of order $\wh p=q+p$ and $\wh\varphi$ has a pole of order $\wh q=q$, so the slope of $\El(\wh\rho,\wh\varphi,\wh R)$ is $q/(q+p)$ (this is well known,\cf \cite{Malgrange91}). We thus have
\begin{itemize}
\item
$\text{slope}^{-1}\cF^{(0,\infty)}\El(\rho,\varphi,R)=1+\text{slope}^{-1}\El(\rho,\varphi,R)=1+p/q$,
\item
$\irr_0\cF^{(0,\infty)}\El(\rho,\varphi,R)=\irr_0\El(\rho,\varphi,R)=qr$,
\item
$\rg\cF^{(0,\infty)}\El(\rho,\varphi,R)=\rg\El(\rho,\varphi,R)+\irr_0\El(\rho,\varphi,R)$.
\end{itemize}
If we choose a $(q+p)$th root $u\mto\lambda(u)$ of $\wh\rho(u)$ and if we denote by $\lambda^{-1}(u)$ its reciprocal series, then, according to \eqref{eq:rholambda}, $\El(\wh\rho,\wh\varphi,\wh R)\simeq\El([u\mto u^{q+p}],\wh\varphi\circ\lambda^{-1},\wh R)$.

\item\label{rem:main3}
Let us also note that twisting $R$ by $L_q$ consists in multiplying its monodromy by $(-1)^q$ (this was denoted by $\otimes (t^{q/2})$ in Proposition~\ref{prop:detel}).

\item\label{rem:main4}
The inverse functor of $\cF^{(0,\infty)}_\pm$ is the local Laplace transform $\cF^{(\infty,0)}_\mp$ from $\Clth$-vector spaces with connection \emph{having slope $<1$} to $\Clt$-vector spaces with connection. If $\sigma$ has degree $p$ and $\psi$ has a pole of order $q<p$, we then have
\begin{equation}\label{eq:Finf0}
\cF^{(\infty,0)}_\pm\El(\sigma,\psi,S)=\El(\pm\sigma^2\psi'/\sigma',\psi+(\sigma/\sigma')\psi',S\otimes L_q).
\end{equation}

\item
The determinant (over $\Clth$) of $\El(\wh\rho,\wh\varphi,\wh R)$ is a rank-one regular connection which is equal to the regular factor of the determinant (over $\Clt$) of $\El(\rho,\varphi,R)$.

Indeed, since $\El(\wh\rho,\wh\varphi,\wh R)$ has slope $<1$, the first assertion is a consequence of Corollary \ref{cor:detreg}. Now, Proposition~\ref{prop:detel} gives $[\det\El(\rho,\varphi,R)]_{\reg}=\det R\otimes (t^{(p-1)r/2})$ and
\[
\det\El(\wh\rho,\wh\varphi,\wh R)=\det\wh R\otimes(t^{(p+q-1)\wh r/2})=\det R\otimes (t^{(p-1)r/2}),
\]
as $\wh r=r$ and $\det\wh R=\det R\otimes(t^{q\wh r/2})$.
\end{enumerate}
\end{remarques}

\begin{proof}[Proof of Theorem \ref{th:main}]
We first choose an algebraic model for $\El(\rho,\varphi,R)$, that is, we assume that $\rho(u)=u^p$, $\varphi(u)=u^{-q}a(u)$ where $a\in\CC[u]$ has degree $<q$ and $a(0)\neq0$; we moreover assume that $R$ is a free $\CC[u,u^{-1}]$-module with a connection having a regular singularity at $u=0$ and $u=\infty$ and no other pole. We now work algebraically and consider that $u$ is the coordinate of the torus $\GG_{m,u}$ and that $\theta$ is also a coordinate of $\GG_{m,\theta}$. We wish to compute the direct image of \hbox{$\big(\CC[\theta,\theta^{-1}]\otimes R\big)\otimes\cE^{\varphi(u)-u^p/\theta}$} by the projection to $\GG_{m,\theta}$, and then take its formal irregular part at $\theta=0$. Theorem~\ref{th:rouc} suggests us to consider the direct image $M$ of $R[\theta,\theta^{-1}]\defin\CC[\theta,\theta^{-1}]\otimes R$ by the morphism \hbox{$\pi:\GG_{m,\theta}\times\GG_{m,u}\to\GG_{m,\theta}\times\Afu_x$} given by
\[
\pi:(\theta,u)\mto(\theta,\varphi(u)-u^p/\theta).
\]
This morphism is finite, so $M$ consists of a single cohomology module, and it is regular holonomic, as $R$ is so. The critical locus of $\pi$ is defined by $\theta=pu^p/u\varphi'(u)=\wh\rho(u)$ and the singular support $S^*$ of $M$ is the curve parametrized by
\[
\GG_{m,u}\ni u\mto\Big(\wh\rho(u),\varphi(u)-\frac{\rho(u)}{\rho'(u)}\,\varphi'(u)\Big).
\]
The germ $\Afu_u$ at the origin is a parametrization of the germ at $(0,\infty)$ of the closure~$S$ of~$S^*$ through the natural extension of this map to $\Afu_u$.

A possible way of proving the theorem, by applying Theorem~\ref{th:rouc} to $M$, would be to compute the vanishing cycle module of $M$ along $S$. Nevertheless, one should first check that Assumption~\ref{assumpt:irred2} is fulfilled, which is not straightforward. So we will not apply this theorem, but we will use the same method without taking the direct image by $\pi$. The previous reasoning is now regarded as a heuristic justification of the formula given in the theorem (at least for $\wh\rho$ and $\wh\varphi$).

We first apply the ramification $\wh\rho:\eta\to\theta$ to $R[\theta,\theta^{-1}]\otimes\cE^{\varphi(u)-u^p/\theta}$ and then take the direct image by the projection to $\GG_{m,\eta}$. In order to compute the regular part with respect to the exponential term $\cE^{\wh\varphi(\eta)}$, we tensor this direct image by $\cE^{-\wh\varphi(\eta)}$ and then take the moderate nearby cycles $\psi_\eta$ of the resulting meromorphic connection. Arguing as in~\cite{Roucairol07}, it is equivalent to tensor with $\cE^{-\wh\varphi(\eta)}$ and to take $\psi_\eta$ first, and then to take the direct image by the projection to the point. We are therefore led to compute $\psi_\eta\big(\cE^{\varphi(u)-\wh\varphi(\eta)-u^p/\wh\rho(\eta)}\otimes R[\eta,\eta^{-1}]\big)$.

We will use the following notation: $\varphi(u)=u^{-q}a(u)$, $u\varphi'(u)=u^{-q}b(u)$ with $b(u)=-qa(u)+ua'(u)$. We also set $a(u)=\sum_{k=0}^{q-1}a_ku^k$, hence
\[
b(u)=\sum_{k=0}^{q-1}(k-q)a_ku^k,\quad a(u)-\frac1p b(u)=\sum_{k=0}^{q-1}\frac{p+q-k}{p}\,a_ku^k.
\]
By assumption, $a_0\neq0$ and if we set $(p,q-k)=d_k$, then $\gcd((d_k)_{k\mid a_k\neq0})=1$. We have
\[
\varphi(u)-\frac{u^p}{\wh\rho(\eta)}-\wh\varphi(\eta)=\frac{1}{u^q\eta^{p+q}}\,\Big(a(u)\eta^{p+q}-u^{p+q}\frac{b(\eta)}{p}-u^q\eta^p(a(\eta)-\tfrac1pb(\eta))\Big).
\]
In order to simplify this expression, we blow up the ideal $(u,\eta)$. Let us consider the chart with coordinates $(v,\eta)$ with $e:(v,\eta)\mto(v\eta,\eta)$. The previous expression becomes
\[
\frac{1}{v^q\eta^q}\,\Big(a(v\eta)-v^{p+q}\frac{b(\eta)}{p}-v^q(a(\eta)-\tfrac1pb(\eta))\Big)=:\frac{f(v,\eta)}{v^q\eta^q}.
\]
We have to compute $\psi_\eta(\cE^{f(v,\eta)/v^q\eta^q}\otimes e^+(R[\eta,\eta^{-1}]))$. This module is supported on $\eta=0$ by definition. Moreover, according to Proposition \ref{prop:bibi}, it is supported at most on the set defined by $f(v,0)=0$. One can also check that it has no component supported at $v=\infty$, by using the same argument in the other chart of the blowing-up space. The function $f$ can be written as
\[
f(v,\eta)=\sum_{k=0}^{q-1}a_k\eta^k\Big(v^k+\frac{q-k}{p}\,v^{p+q}-\frac{p+q-k}{p}\,v^q\Big).
\]
It is easy to compute that the polynomial $f(v,0)=a_0(1+\frac qpv^{p+q}-(1+\frac qp)v^q)$ has exactly $d_0=(p,q)$ double roots, which are the $d_0$th roots of unity, and the other roots $v_i$ are simple. The branches of $f(v,\eta)$ at the points $(v_i,0)$ are thus smooth and transversal to $\eta=0$. Moreover, one easily checks that $(v-1)^2$ divides $f(v,\eta)$, so that in some neighbourhood of $(1,0)$, $f(v,\eta)=(v-1)^2\cdot\text{unit}$. Lastly, for any~$\zeta$ with $\zeta^{d_0}=1$ and $\zeta\neq1$, if we set $v'=v-\zeta$, the germ of $f$ at $(\zeta,0)$ can be written as $v^{\prime2}\lambda(v',\eta)+\eta^\ell\mu(v',\eta)$, where $\lambda,\mu$ are units and $\ell\in[1,q-1]$: indeed, the assumption implies that, for any such $\zeta$, there exists a smallest $\ell\in[1,q-1]$ such that $a_\ell\neq0$ and $\zeta^\ell\neq1$; we have $f(\zeta,\eta)=\sum_ka_k(\zeta^k-1)\eta^k$, so the first nonzero term in this series is for $k=\ell$; similarly, the coefficient of $v'$ in the expansion of~$f$ is $\sum_kka_k(\zeta^k-1)\eta^k$, which is a multiple of $\eta^\ell$. The situation is illustrated in Figure~\ref{fig:2}.
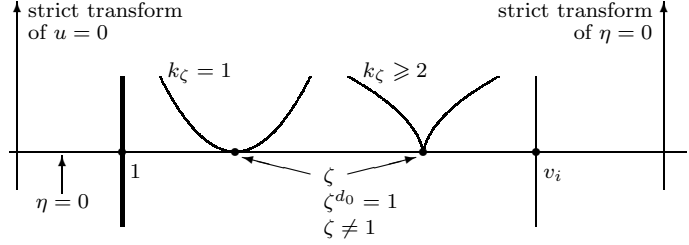
\begin{figure}[htb]
\begin{center}
\begin{picture}(90,30)(0,0)
\put(0,10){\line(1,0){90}}
\put(70,0){\line(0,1){20}}
\put(70,10){\circle*{1}}
\put(71,6.5){\footnotesize $v_i$}
\qbezier(20,20)(30,0)(40,20)
\put(21,20){\footnotesize $k_\zeta=1$}
\put(30,10){\circle*{1}}
\put(40,2.5){\footnotesize $\begin{array}{l}\zeta\\\zeta^{d_0}=1\\ \zeta\neq1\end{array}$}
\put(45,7){\vector(4,1){9}}
\put(40,7){\vector(-4,1){9}}
\qbezier(55,10)(54,15)(45,20)
\qbezier(55,10)(56,15)(65,20)
\put(47,20){\footnotesize $k_\zeta\geq2$}
\put(55,10){\circle*{1}}
\put(1,5){\vector(0,1){25}}
\put(87,5){\vector(0,1){25}}
\put(2.5,28){{\footnotesize strict transform}}
\put(2.5,25){{\footnotesize of $u=0$}}
\put(65,28){{\footnotesize strict transform}}
\put(75,25){{\footnotesize of $\eta=0$}}
\put(3.5,2.5){\footnotesize $\eta=0$}
\put(7,4.5){\vector(0,1){5}}

\linethickness{1.3pt}
\put(15,0){\line(0,1){20}}
\put(15,10){\circle*{1}}
\put(16,6.5){\footnotesize $1$}
\end{picture}
\end{center}
\caption{The branches of curve $f(v,\eta)=0$ near $\eta=0$}\label{fig:2}
\end{figure}

In suitable local coordinates $(v',\eta)$ centered at the points $(v_i,0)$ or the points $(\zeta,0)$ with $\zeta^{d_0}=1$, the exponent $f(v,\eta)/v^q\eta^q$ can thus be written as $g(v',\eta)/\eta^q$, where
\[
g(v',\eta)=
\begin{cases}
v'&\text{at }(v_i,0),\\
v^{\prime2}&\text{at }(1,0),\\
v^{\prime2}\lambda(v',\eta)+\eta^\ell\mu(v',\eta) \text{ (with $\ell\in[1,q-1]$)} &\text{at }(\zeta,0).
\end{cases}
\]

Theorem \ref{th:main} is now a consequence of Lemma \ref{lem:psi} below:
we apply the lemma to $S=e^+(R[\eta,\eta^{-1}])$, \ref{lem:psi}\eqref{lem:psi1} to the germ of~$f$ at each $(v_i,0)$, \ref{lem:psi}\eqref{lem:psi2} to the germ of~$f$ at $(1,0)$ and \ref{lem:psi}\eqref{lem:psi3} to the germ of~$f$ at at each $(\zeta,0)$ with $\zeta^{d_0}=1$ and $\zeta\neq1$.
\end{proof}

\begin{lemme}\label{lem:psi}
Let $S$ be a germ of regular meromorphic connection in coordinates $(v',\eta)$ with pole on $\eta=0$ at most and let $q\geq1$. Then
\begin{enumerate}
\item\label{lem:psi1}
$\psi_\eta(\cE^{v'/\eta^q}\otimes S)=0$.
\item\label{lem:psi2}
$\psi_\eta(\cE^{v^{\prime2}/\eta^q}\otimes S)$ is supported at $v'=0$ and its germ is isomorphic to the germ at $v'=0$ of $\psi_\eta(S\otimes L_{\eta,q})$ with $L_{\eta,q}\simeq(\CC[\eta,\eta^{-1}],d-\tfrac q2\,\tfrac{d\eta}{\eta})$.
\item\label{lem:psi3}
Let $h(v',\eta)=v^{\prime2}\lambda(v',\eta)+\eta^\ell\mu(v',\eta)$, where $\lambda,\mu$ are local units and $\ell\in\NN^*$. Let us assume that $q\geq\ell+1$. Then $\psi_\eta(\cE^{h(v',\eta)/\eta^q}\otimes S)=0$.
\end{enumerate}
\end{lemme}

\begin{proof}\mbox{}\par
\eqref{lem:psi1}
As $S$ is an iterated extension of rank-one meromorphic connections, we can assume that it is isomorphic to $\CC\{\eta\}[\eta^{-1}]$ with connection $d-\alpha d\eta/\eta$. Then $\cE^{v'/\eta^q}\otimes S\simeq \CC\{\eta\}[\eta^{-1}]$ with connection $d+dv'/\eta-(\alpha+v'/\eta) d\eta/\eta$. The generator $\epsilon=1$ satisfies $\epsilon=\eta\partial_{v'}\epsilon$, showing that $\psi_\eta(\cE^{v'/\eta^q}\otimes S)=0$.

\eqref{lem:psi2}
The first part of the assertion is clear. For the second part, we argue by induction on $q$. Let us first assume that $q\geq3$. We blow up the ideal $(v',\eta)$. We get two charts: $(v'_1,\eta_1)$ with $v'=v'_1\eta_1$ and $\eta=\eta_1$ and $(v'_2,\eta_2)$ with $v'=v'_2$ and $\eta=v'_2\eta_2$.

In the chart $(v'_2,\eta_2)$, we find $\psi_{v'_2\eta'_2}(\cE^{1/v_2^{\prime q-2}\eta_2^q}\otimes e^+S)=0$. In the chart $(v'_1,\eta_1)$, $\psi_{\eta_1}(\cE^{v_1^{\prime2}/\eta_1^{q-2}}\otimes e^+S)$ is supported at the origin of the chart, and we conclude by induction.

If $q=2$, we identify (in the same way as we did in the proof of Proposition \ref{prop:rouc}) $\psi_{\eta\circ e}e^+(\cE^{v^{\prime2}/\eta^2}\otimes S)$ (with its automorphism $T$ coming with the functor $\psi_{\eta\circ e}$) with the $\CC[v'_1]$-free module with connection $(\psi_\eta S,T)\otimes_\CC\cE^{v_1^{\prime2}}$, where $\cE^{v_1^{\prime2}}=(\CC[v'_1],d+\nobreak2v'_1)$. It is then enough to show that the direct image by $e$ of $(\CC[v'_1],d+2v'_1)$ is equal to $\CC$, that is, that the morphism $\partial_{v'_1}+2v'_1:\CC[v'_1]\to\CC[v'_1]$ is injective and has a cokernel equal to $\CC$. This is easily checked.

Let us end with the case $q=1$. We first show that the dimension of the space $\psi_\eta(\cE^{v^{\prime2}/\eta}\otimes\nobreak S)_0$ is equal to that of $(\psi_\eta S)_0$ by a computation of Euler characteristics, as in \cite{Roucairol07}. In order to do this computation, we first blow up the ideal $(v',\eta)$ and then, in the chart $(v'_2,\eta_2)$, we blow up the ideal $(v'_2,\eta_2)$. We denote by $E_2\simeq\PP^1$ the exceptional divisor over $v'_2=0,\eta_2=0$ and by $E_1\simeq\PP^1$ the strict transform of the exceptional divisor obtained after the first blowing-up. After the total blow up $\wt e$, the module $\psi_{\eta\circ\wt e}\wt e{}^+(\cE^{v^{\prime2}/\eta}\otimes S)$ is supported on $E_1\cup E_2$ and $e{}^+(\cE^{v^{\prime2}/\eta}\otimes S)$ has regular singularity along $E_1\cup E_2$ except at $P_2$. Figure \ref{fig:3} gives the values of the function $x\mto\chi\DR\psi_{\eta\circ\wt e}\wt e{}^+(\cE^{v^{\prime2}/\eta}\otimes S)_x$ on each stratum of the natural stratification of $\wt e{}^{-1}(\{\eta=0\})$, when $\rg S=1$ (which is clearly enough), from which one deduces the desired assertion, as
\[
-1\cdot\chi_\top(E_1\moins\{P_1\})+0\cdot\chi_\top(P_1)-2\cdot\chi_\top(E_2\moins\{P_1,P_2\})+2\cdot\chi_\top(P_2)=1.
\]

\begin{figure}[htb]
\begin{center}
\begin{picture}(60,25)(0,0)
\put(0,10){\line(1,0){60}}
\put(45,0){\vector(0,1){20}}
\put(45,10){\circle*{1}}
\put(47,18){{\footnotesize strict transform}}
\put(47,15){{\footnotesize of $\eta=0$}}

\put(15,0){\line(0,1){25}}
\put(15,10){\circle*{1}}
\put(10.5,23){{\footnotesize $E_1$}}
\put(1,11.5){{\footnotesize $E_2$}}
\put(16,11.5){{\footnotesize $P_1$}}
\put(40.5,11.5){{\footnotesize $P_2$}}

\put(18.5,24){\footnotesize $\chi=-1$}
\put(22,23){\vector(-1,-1){5.5}}

\put(5,.5){\footnotesize $\chi=0$}
\put(8.5,3.3){\vector(1,1){5.5}}

\put(26,.5){\footnotesize $\chi=-2$}
\put(30,3.3){\vector(0,1){6}}

\put(50,.5){\footnotesize $\chi=2$}
\put(52,3.3){\vector(-1,1){5.5}}

\end{picture}
\end{center}
\caption{\hbox{The Euler characteristic function of $\DR\psi_{\eta\circ\wt e}\wt e{}^+(\cE^{v^{\prime2}/\eta}\otimes S)$}}\label{fig:3}
\end{figure}

Once this computation is done, we prove the result for $S=\CC\{\eta\}[\eta^{-1}]^{\rg S}$ with connection $d+(\alpha\Id+N)d\eta/\eta$, where $N$ is a constant nilpotent matrix. Then \hbox{$\cE^{v^{\prime2}/\eta}\otimes S$} has a basis $\epsilon$ which satisfies
\[
\eta\partial_{v'}\epsilon=2v'\epsilon\quad\text{and}\quad \eta\partial_\eta\epsilon=\epsilon\cdot\big((\alpha-v^{\prime2}/\eta)\Id+N\big),
\]
from which we deduce, as $\eta\partial_{v'}^2\epsilon=(2+4v^{\prime2}/\eta)\epsilon$,
\begin{equation}\label{eq:N}
\eta\partial_\eta\epsilon=\epsilon\cdot\big((\alpha+1/2)\Id+N\big)-\frac14\eta\partial_{v'}^2\epsilon.
\end{equation}
A standard computation of $V$-filtration now shows that each $\epsilon_i\in\epsilon$ has order $\alpha+1/2$ with respect to the $V$-filtration and that the classes of the $\epsilon_i$ in $\psi_\eta(\cE^{v^{\prime2}/\eta}\otimes S)_0$ generate this $\CC$-vector space. By the dimension count above, they form a basis of this space. From \eqref{eq:N} one then gets that
\[
\psi_\eta(\cE^{v^{\prime2}/\eta}\otimes S)_0\simeq\big(\CC^{\rg S},\exp2i\pi((\alpha+1/2)\Id+N)\big),
\]
as was to be proved.

\eqref{lem:psi3}
We argue by induction on $\ell\geq1$. Let us first assume that $\ell\geq3$ (hence $q\geq3$) and let us blow up the ideal $(v',\eta)$. In the chart $(v'_1,\eta_1)$, $\psi_{\eta_1}e^+(\cE^{h(v',\eta)/\eta^q}\otimes S)=\psi_{\eta_1}(\cE^{h_1(v'_1,\eta_1)/\eta_1^{q-2}}\otimes e^+S)$ with $h_1$ of the same form as $h$, but with $\ell$ replaced by $\ell-2$, hence the result by induction.

If $\ell=2$, the strict transform of $h=0$ cuts the exceptional divisor $\eta_1=0$ transversally at two distinct points. Applying \eqref{lem:psi1} at these points and an easy argument at the other points $(v'_1,0)$, we obtain $\psi_{\eta_1}e^+(\cE^{h(v',\eta)/\eta^q}\otimes S)=0$. In the chart $(v'_2,\eta_2)$, we are led to compute $\psi_{v'_2\eta_2}(\cE^{\text{unit}/v_2^{\prime q-2}\eta^q}\otimes e^+S)$ and, since $q\geq3$, we can apply Proposition \ref{prop:bibi} to obtain that the result is $0$.

We now assume that $\ell=1$, so $q\geq2$, and we blow up as above. We find that $\psi_{\eta\circ e}e^+(\cE^{h/\eta^q}\otimes S)$ is supported at the origin of the chart $(v'_2,\eta'_2)$, and is equal to $\psi_{v'_2\eta_2}(\cE^{h_2/v_2^{\prime q-1}\eta_2^q}\otimes e^+S)$, with $h_2(v'_2,\eta_2)=v'_2\lambda(v'_2,v'_2\eta_2)+\eta_2\mu(v'_2,v'_2\eta_2)$. If we blow up the center of this chart, we find that the strict transform of $h_2=0$ is smooth and cuts transversally the exceptional divisor. Applying \eqref{lem:psi1} at this point and Proposition \ref{prop:bibi}\eqref{prop:bibi2} at the crossing points of the pull back of the divisor $v'_2\eta_2=0$ gives the desired vanishing.
\end{proof}

\begin{remarque}[Extension to $\wh\cD$-modules, \cf \cite{Malgrange91}]\label{rem:F0infDhat}
Let $M$ be a $\wh\cD$-module (\cf Remark \ref{rem:reg}). The local Laplace transform $\cF^{(0,\infty)}_\pm$ is naturally extended to such objects, in such a way that, if we choose a module $N$ on the Weyl algebra with regular singularities at $0$ and $\infty$ and no other singularities, and such that $\Ct\otimes_{\Cht}N=M_\reg$, then $\cF^{(0,\infty)}_\pm M_\reg$ is the germ at infinity of the Laplace transform of $N$.

It is known that giving $M_\reg$ is equivalent to giving a diagram of vector spaces with automorphisms $T$
\[
\xymatrix@C=1cm{
\psi_tM_\reg\ar@/^1pc/[r]^{c}&\phi_tM_\reg\ar@/^1pc/[l]^{v}
}
\]
where $(\psi_tM_\reg,T)=(\psi_tM_\reg,\Id+vc)$ and $(\phi_tM_\reg,T)=(\phi_tM_\reg,\Id+cv)$. Let us remark that the effect of the involution $\iota^+$ is to replace $c$ with $-c$ and $v$ with $-v$, so that $\iota^+M_\reg\simeq M_\reg$.

Giving the vector space with automorphism $(\psi_tM_\reg,T)$ is equivalent to giving a regular $\Clt$-module with connection: this is $\Clt\otimes_{\Ct}M_\reg$.

On the other hand, the $\Clth$-vector space with connection $\cF^{(0,\infty)}_\pm M_\reg$ corresponds to the vector space with automorphism $(\phi_tM_\reg,T)$.
\end{remarque}

\subsection{Example of a direct computation}\label{sec:example}
Let us compute directly the local Laplace transform of $\cE^{a/t^q}$ for $a\in\CC^*$ and $q\in\NN^*$. We can define the corresponding $\CC[t]\langle\partial_t\rangle$-module by the differential equation $t^qt\partial_t+qa$. The Laplace transform in the $\theta$-variable (\ie corresponding to the kernel $e^{-t/\theta}$) is obtained by setting $t=\theta^2\partial_\theta$ and $\partial_t=\theta^{-1}$. Using that $\theta^2\partial_\theta\theta^{-1}=\theta\partial_\theta-1$, it has the equation $(\theta^2\partial_\theta)^q(\theta\partial_\theta-1)+qa$. The Newton polygon of this operator at $\theta=0$ has only one slope, which is $q/(q+1)$, so a ramification of order $q+1$ is needed to get an integral slope. Let us set $\theta=-\eta^{q+1}/qa$ (hence $\theta\partial_\theta=(q+1)^{-1}\eta\partial_\eta$). The pull-back by the ramification $\eta\mto-\eta^{q+1}/qa$ of the previous operator is, up to a constant,
\enlargethispage{\baselineskip}%
\begin{multline*}
(\eta^{q+1}\eta\partial_\eta)^q(\eta\partial_\eta-(q+1))+(-1)^q[q(q+1)a]^{q+1}\\[-5pt]
=\prod_{k=1}^{q+1}(\eta^{q+1}\partial_\eta-k\eta^q)+(-1)^q[q(q+1)a]^{q+1}.
\end{multline*}
We twist the corresponding meromorphic connection with $\cE^{-\lambda/\eta^q}$ in order to create a regular part. After twisting, the new operator is obtained from the old one by replacing $\eta^{q+1}\partial_\eta$ with $\eta^{q+1}\partial_\eta-q\lambda$, that is,
\[
\prod_{k=1}^{q+1}(\eta^{q+1}\partial_\eta-q\lambda-k\eta^q)+(-1)^q[q(q+1)a]^{q+1}.
\]
We choose $\lambda$ such that the constant term vanishes, that is, $\lambda=\zeta(q+1)a$ with $\zeta^{q+1}=1$. Let us fix $\zeta=1$ for instance. The resulting operator can now be divided by $\eta^q$ and gets as regular part the following operator:
\[
[-q(q+1)a]^q(q+1)\Big[\eta\partial_\eta-\frac{q+2}{2}\Big].
\]
In other words, if we denote by $L_q$ the rank-one local system with monodromy $(-1)^q$, we find that the localized Laplace transform of $\cE^{a/t^q}$ is $\El(\wh\rho,\wh\varphi,L_q)$, where $\wh\rho,\wh\varphi$ are given by the theorem.

\subsection{The local Laplace transforms $\cF^{(s,\infty)}$ and $\cF^{(\infty,\infty)}$}\label{subsec:cinf-infinf}

\subsubsection*{$\cF^{(s,\infty)}$, $s\in\CC$}
The local Laplace transform $\cF^{(s,\infty)}_\pm$ is defined as $\cE^{\pm s/\theta}\otimes\cF^{(0,\infty)}_\pm$. Therefore, the formula for computing $\cF^{(s,\infty)}_\pm\El(\rho,\varphi,R)$ is straightforwardly obtained from that giving $\cF^{(0,\infty)}_\pm\El(\rho,\varphi,R)$:\begin{align}
&\cF^{(s,\infty)}_\pm\El(\rho,\varphi,R)\simeq\El(\wh\rho_\pm,\wh\varphi\pm c/(\theta\circ\wh\rho),\wh R)\quad \text{if $\varphi\not\in\Cu$},\\
&\cF^{(s,\infty)}_\pm M_\reg\simeq \El(\Id,\pm c/\theta,\cF^{(0,\infty)}_\pm M_\reg)\quad\text{($M_\reg$ a regular $\wh\cD$-module).}
\end{align}
When $s\neq0$, $\cF^{(s,\infty)}_\pm M$ has slope one for any holonomic $\wh\cD$-module.

\subsubsection*{$\cF^{(\infty,\infty)}$}
The transform $\cF^{(\infty,\infty)}_\pm$ corresponds to the integral transform $\int\cbbullet e^{\pm1/t\theta}(-dt/t^2)$. It applies only to $\Clt$-vector spaces with connection having slope $>1$ and produces a $\Clth$-vector spaces with connection having slope $>1$.

If we consider an elementary formal connection $\El(\rho,\varphi,R)$ with $q>p$ then, setting now
\begin{equation}\label{eq:infinf}
\wh\rho_\pm(u)=\pm\frac{\rho'}{\varphi'\rho^2},\quad\wh\varphi=\varphi+\frac{\rho(u)}{\rho'(u)}\,\varphi'(u),\quad\wh R=R\otimes L_q,
\end{equation}
we have $\cF^{(\infty,\infty)}_\pm\El(\rho,\varphi,R)\simeq\El(\wh\rho_\pm,\wh\varphi,\wh R)$. The proof is similar to that of Theorem~\ref{th:main} and we will not repeat it.

\begin{remarques}\mbox{}
\begin{enumerate}
\item
If $q>p$, we have $\wh p=q-p$ and $\wh q=q$. Moreover,
\begin{itemize}
\item
$\text{slope}^{-1}\cF^{(\infty,\infty)}\El(\rho,\varphi,R)=1-\text{slope}^{-1}\El(\rho,\varphi,R)=1-p/q$,
\item
$\irr_0\cF^{(\infty,\infty)}\El(\rho,\varphi,R)=\irr_0\El(\rho,\varphi,R)=qr$,
\item
$\rg\cF^{(\infty,\infty)}\El(\rho,\varphi,R)=\irr_0\El(\rho,\varphi,R)-\rg\El(\rho,\varphi,R)=(q-p)r$.
\end{itemize}
\item
The inverse transform of $\cF^{(\infty,\infty)}_\pm$ is $\cF^{(\infty,\infty)}_\mp$.
\end{enumerate}
\end{remarques}

\subsection{The use of the index of rigidity}\label{subsec:rigidity}
Let $X$ be a smooth projective curve, let $S$ be a finite set of point and let $\cM(*S)$ be a locally free $\cO_X(*S)$-module of rank $r$ with connection $\nabla$. Let $\cM(*S)_{\min}$ be the minimal extension of $\cM(*S)$ (also called middle extension, or intermediate extension) in the sense of holonomic $\cD_X$-modules: $\cM(*S)_{\min}$ is the smallest $\cD_X$-submodule $\cN$ of $\cM(*S)$ such that $\cM(*S)/\cN$ is supported in a finite set. Let us recall the global index formula for $\cM(*S)_{\min}$ (the global index formula for $\cM(*S)$ is obtained by forgetting the last sum in the formula below, and one can refer to \cite[Th\ptbl 4.9(ii), p\ptbl 70]{Malgrange91} for it; the formula for $\cM(*S)_{\min}$ is an easy consequence of it):
\begin{multline}\label{eq:index}
\chi(X,\DR\cM(*S)_{\min})\\
= r\chi_{\top}(X\moins S)+\sum_{s\in S}\irr_s\cM(*S)+\sum_{s\in S}\dim\ker(T_{\reg,s}-\Id),
\end{multline}
where $T_{\reg,s}$ denotes the formal regular monodromy at $s$, that is, the monodromy of $[\wh\cO_{X,s}\otimes\cM(*S)]_\reg$.

Applying this formula to $\End(\cM(*S))$ instead of $\cM(*S)$ gives, according to \eqref{eq:endreg} and to the formula in Remark~\ref{rem:Zrho}:
\begin{multline}\label{eq:indexend}
\chi\big(X,\DR\End(\cM(*S))_{\min}\big)\\= r^2\chi_{\top}(X\moins S)+\sum_{s\in S}\irr_s\End(\cM(*S))+\sum_{s\in S}\sum_{i\in I'_s} p_{s,i}\dim\rZ(T_{s,i}),
\end{multline}
where $\bigoplus_{i\in I'_s}\El(\rho_{s,i},\varphi_{s,i},R_{s,i})$ is the refined Turrittin-Levelt decomposition of $\cM(*S)$ at~$s$, $p_{s,i}$ is the degree of $\rho_{s,i}$, $T_{s,i}$ is the monodromy of $R_{s,i}$ and $\rZ(\cbbullet)$ denotes the centralizer of $\cbbullet$. We will set $I'_s=I_s\cup\{\reg\}$ with an obvious meaning, and $\rho_{s,\reg}(t)=t$, $\varphi_{s,\reg}=0$. Let us note that \eqref{eq:indexend} was yet obtained (in a simpler case) by A.~Paiva \cite{Paiva06}. The left-hand term in \eqref{eq:indexend} is by definition the index of rigidity of $\cM(*S)$ and is denoted by $\rig\cM(*S)$ (\cf \cite{Katz96,B-E04}).

\begin{exemples}\mbox{}
\begin{itemize}
\item
Assume $S=\{0,\infty\}$. Then $\rig\cM(*S)=2$ if and only if, for $s=0,\infty$, $\irr_s=0$, $\# I'_s=1$ and $\dim\rZ(\rho_{s,+}T_{s})=1$. This is equivalent to ask that $\cM(*S)$ has rank one.
\item
(Compare with \cite[Cor\ptbl4.9]{B-E04}.) Assume $\cM(*S)$ has rank one. Then, for any $s\in S$, $\irr_s=0$, $\# I'_s=1$ and $\dim\rZ(\rho_{s,+}T_{s})=1$. Therefore, $\rig\cM(*S)=\chi_{\top}(\PP^1\moins S)+\#S=\chi_{\top}(\PP^1)=2$.
\end{itemize}
\end{exemples}

Let us now assume that $X=\PP^1$ and that $S$ contains $\infty$. We then set $M_{\min}=\Gamma\big(\PP^1,[\cM(*S)_{\min}](*\infty)\big)$, which is a holonomic module over the Weyl algebra $\CC[t]\langle\partial_t\rangle$. Let $F_\pm(M_{\min})$ be its Laplace transform with respect to the kernel $e^{\pm t\tau}$. This is a holonomic module over the Weyl algebra $\CC[\tau]\langle\partial_\tau\rangle$ (\cf the introduction). We denote by $\wh S\subset\wh\PP^1$ the set of its singularities (including $\wh\infty$) on the $\tau$-line. There exists a unique $\cD_{\wh\PP^1}$-module $\wh\cM_\pm$, equal to $\wh\cM_\pm(*\wh\infty)$, such that $F_\pm(M_{\min})=\Gamma(\wh\PP^1,\wh\cM_\pm)$. We then consider the associated bundle with connection $\wh\cM_\pm(*\wh S)$. If $M$ is irreducible (or semi-simple), then $F_\pm M$ is so, and both are equal to their minimal extensions. Let us recall:

\begin{theoreme}[\cite{Katz96,B-E04}]\label{th:rig}
If $M$ is irreducible, then $\rig\wh\cM_\pm(*\wh S)=\rig\cM(*S)$.\qed
\end{theoreme}

We will show that this theorem implies, when $M$ is irreducible, a relation between the dimensions of the centralizers of the formal monodromies corresponding to the purely irregular parts at the singularities of $M$.

Let $\psi$ be a finite dimensional vector space equipped with an automorphism ${}^\psi\!T$ and let us set $\phi=\im({}^\psi\!T-\Id)$, equipped with the induced automorphism ${}^\phi\!T$. Then (\cf \eg \cite[Prop\ptbl2.4.10]{Paiva06})
\begin{equation}\label{eq:Zmin}
\dim \rZ({}^\psi\!T)-\dim \rZ({}^\phi\!T)=\big(\dim\ker({}^\psi\!T-\Id)\big)^2=:\kappa^2.
\end{equation}

Let us take the notation given after \eqref{eq:indexend} and let us assume $M=M_{\min}$. The formal stationary phase formula of \cite{B-E04,Garcia04} implies, for the part with slope $\leq1$ at $\wh\infty$, and setting $F=F_-$ for instance:
\begin{equation}
[(FM)(*\wh S)]^{\leq1}_{\wh\infty}\simeq\tbigoplus_{s\in S\moins\{\infty\}}\Big[(\cE^{-s/\theta}\otimes M_{s,\reg}^\phi)\oplus\tbigoplus_{i\in I_s}\El(\wh{\rho_{s,i}},\wh{\varphi_{s,i}},\wh{R_{s,i}})\Big],
\end{equation}
where $M_{s,\reg}^\phi$ corresponds to the monodromy ${}^\phi\!T_{s,\reg}$, if ${}^\psi\!T_{s,\reg}=T_{s,\reg}$ is defined as in \eqref{eq:index} and $\wh{\rho_{s,i}}$ is obtained from the formula in Theorem \ref{th:main} of \S\ref{subsec:cinf-infinf} for $\cF^{(s,\infty)}$. Indeed, as explained in the introduction, and according to Theorem \ref{th:main}, it remains to justify the terms $\cE^{-s/\theta}\otimes M_{s,\reg}^\phi$ and, after a translation to the origin, it is enough to consider the case where $s=0$; in other words, one is lead to compute the formal Fourier transform at $\wh\infty$ of a regular minimal extension formal $\cD$-module; an easy computation gives the desired formula.

Similarly, with obvious notation,
\begin{equation}
[(FM)(*\wh S)]^{>1}_{\wh\infty}\simeq\tbigoplus_{i\in I_\infty^{>1}}\El(\wh{\rho_{\infty,i}},\wh{\varphi_{\infty,i}},\wh{R_{\infty,i}}).
\end{equation}
If we denote by $Z$ the last sum in \eqref{eq:indexend} and by $\wh Z$ the last sum in the formula $\wh{\eqref{eq:indexend}}$ obtained by applying \eqref{eq:indexend} to $FM$, we find, assuming $M=M_{\min}$ and $FM=(FM)_{\min}$, and using \eqref{eq:Zmin}:
\begin{align*}
Z&=\sum_{s\in S\moins\{\infty\}}\Big[\dim\rZ(T_{s,\reg})+\sum_{i\in I_s}p_{s,i}\dim \rZ(T_{s,i})\Big]\\[-5pt]
&\hspace*{1cm}+\sum_{\wh s\in \wh S\moins\{\wh\infty\}}\Big[\dim\rZ({}^\phi\!T_{\wh s,\reg})+\sum_{i\in I_{\wh s}}\wh{p_{\wh s,i}}\dim \rZ(T_{\wh s,i})\Big]+\sum_{i\in I_\infty^{>1}}p_{\infty,i}\dim \rZ(T_{\infty,i})\\
\wh Z&=\sum_{s\in S\moins\{\infty\}}\Big[\dim\rZ({}^\phi\!T_{s,\reg})+\sum_{i\in I_s}\wh{p_{s,i}}\dim \rZ(T_{s,i})\Big]\\[-5pt]
&\hspace*{1cm}+\sum_{\wh s\in \wh S\moins\{\wh\infty\}}\Big[\dim\rZ(T_{\wh s,\reg})+\sum_{i\in I_{\wh s}}p_{\wh s,i}\dim \rZ(T_{\wh s,i})\Big]+\sum_{i\in I_\infty^{>1}}\wh{p_{\infty,i}}\dim \rZ(T_{\infty,i}),
\end{align*}
and, setting $\kappa_{s,\reg}=\dim\ker(T_{s,\reg}-\Id)$, using that $p_{s,i}-\wh{p_{s,i}}=q_{s,i}=\wh{q_{s,i}}$ if $s\in S\moins\{\infty\}$, a corresponding equality at $\wh s$, and $p_{\infty,i}+\wh{p_{\infty,i}}=q_{\infty,i}=\wh{q_{\infty,i}}$,
\begin{multline}\label{eq:ZZhat}
Z-\wh Z=\sum_{s\in S\moins\{\infty\}}\Big[\kappa_{s,\reg}^2+\sum_{i\in I_s}q_{s,i}\dim \rZ(T_{s,i})\Big]\\[-5pt]
-\sum_{\wh s\in \wh S\moins\{\wh\infty\}}\Big[\kappa_{\wh s,\reg}^2+\sum_{i\in I_{\wh s}}q_{\wh s,i}\dim \rZ(T_{\wh s,i})\Big]\\[-5pt]
+\sum_{i\in I_\infty^{>1}}(2p_{\infty,i}-q_{\infty,i})\dim \rZ(T_{\infty,i}).
\end{multline}

When $M$ is irreducible, the relation $\chi-\wh\chi=0$ given by Theorem \ref{th:rig} leads to an expression of $Z-\wh Z$ in terms \emph{not} depending on the formal monodromies, that~is,
\[
Z-\wh Z=r^2(\#S-2)-\wh r{}^2(\#\wh S-2)-\sum_{s\in S}\irr_s\End(\cM(*S))+\sum_{\wh s\in\wh S}\irr_{\wh s}\End(\wh\cM(*\wh S)),
\]
hence the desired relation, by combining with \eqref{eq:ZZhat}.

\begin{remarque}
On the other hand, if we assume that $I_s$, $I_{\wh s}$ and $I_\infty^{>1}$ are all empty (\ie $M$ regular at finite distance and having at $t=\infty$ a formal decomposition $\bigoplus_{\wh s\in\wh S\moins \wh\infty}(\cE^{\wh st}\otimes M_{\wh s,\reg}^\phi)$), the previous computation can be used to give a proof of Theorem \ref{th:rig}, \cf \cite{Paiva06}.
\end{remarque}

\providecommand{\bysame}{\leavevmode\hbox to3em{\hrulefill}\thinspace}
\providecommand{\MR}{\relax\ifhmode\unskip\space\fi MR }
\providecommand{\MRhref}[2]{%
  \href{http://www.ams.org/mathscinet-getitem?mr=#1}{#2}
}
\providecommand{\href}[2]{#2}

\end{document}